\documentclass[10pt,a4paper]{amsart}
\usepackage[margin=3cm]{geometry}
\usepackage{amssymb}
\usepackage{graphicx} 
\usepackage{mathrsfs}
\usepackage{enumerate}
\usepackage{xspace}
\usepackage{color}
\usepackage{enumerate}
\usepackage{enumitem}
\usepackage{amsthm}
\usepackage{dsfont}
\usepackage{bbm}
\usepackage[colorlinks=true, linkcolor = blue, citecolor = blue]{hyperref}
\usepackage[numbers]{natbib}
\usepackage[backgroundcolor=white,bordercolor=red]{todonotes}
\DeclareMathAlphabet{\mathpzc}{OT1}{pzc}{m}{it}
\usepackage[nameinlink]{cleveref}
\numberwithin{equation}{section}
\begin{document}

\theoremstyle{plain}
\newtheorem*{convention}{Convention}
\newtheorem{theorem}{Theorem}[section]
\newtheorem{lemma}[theorem]{Lemma}
\newtheorem{proposition}[theorem]{Proposition}
\newtheorem{corollary}[theorem]{Corollary}
\newtheorem{definition}[theorem]{Definition}
\theoremstyle{definition}
\newtheorem{remark}[theorem]{Remark}
\newtheorem{listing}[theorem]{Assumption}

\newtheorem{example}[theorem]{Example}

\crefname{lemma}{lemma}{lemmas}
\Crefname{lemma}{Lemma}{Lemmata}
\crefname{corollary}{corollary}{corollaries}
\Crefname{corollary}{Corollary}{Corollaries}
\crefname{example}{example}{examples}
\Crefname{example}{Example}{Examples}
\crefname{listing}{assumption}{assumptions}
\Crefname{listing}{Assumption}{Assumptions}

\makeatletter
\@namedef{subjclassname@2020}{%
	\textup{2020} Mathematics Subject Classification}
\makeatother

\newcommand{\1}{\mathbb{I}}
\newcommand{\X}{\mathsf{X}} 
\newcommand{\D}{{\mathbb{D}}}
\newcommand{\Df}{\mathbf{D}}
\renewcommand{\epsilon}{\varepsilon}
\newcommand{\B}{\mathbb{B}}
\newcommand{\C}{\mathbb{C}}
\newcommand{\Z}{\mathsf{Z}}
\renewcommand{\t}{\mathfrak{t}}
\newcommand{\s}{\mathfrak{s}}
\newcommand{\f}{\mathfrak{f}}
\newcommand{\g}{\mathfrak{g}}
\newcommand{\cH}{\mathcal{H}}
\newcommand{\cF}{\mathcal{F}}

\makeatletter
\newcommand{\mylabel}[2]{#2\def\@currentlabel{#2}\label{#1}}
\makeatother

\renewcommand{\mathcal}{\mathscr}

\title[Propagation of Chaos for Interacting SPDEs]{Propagation of Chaos for Weakly Interacting Mild Solutions to Stochastic Partial Differential Equations} 
\author[D. Criens]{David Criens}
\address{University of Freiburg, Ernst-Zermelo-Str. 1, 79104 Freiburg, Germany}
\email{david.criens@stochastik.uni-freiburg.de}

\keywords{Interacting Particle System, Propagation of Chaos, McKean--Vlasov Equation, SPDE, Weak Solution, Martingale Solution, Uniqueness in Law, Pathwise Uniqueness, Compact Semigroup, Factorization Method, Compactness Method\vspace{1ex}}


\thanks{The author is grateful to an anonymous person for helpful remarks which lead to the formulation of \ref{A3}.}

\date{\today}
\maketitle

\frenchspacing
\pagestyle{myheadings}

\begin{abstract}
This article investigates the propagation of chaos property for weakly interacting mild solutions to semilinear stochastic partial differential equations whose coefficients might not satisfy Lipschitz conditions.
Furthermore, we derive continuity and linear growth conditions for the existence and uniqueness of mild solutions to SPDEs with distribution dependent coefficients, so-called McKean--Vlasov SPDEs.
\end{abstract}

\section{Introduction}
A system of particles can be modeled as interacting stochastic processes. When the number \(N\) of particles gets large the process level usually contains too much information for a statistical description and it is interesting to change the point of view by passing to the macroscopic picture which means looking at the system on an average level. More specifically, the idea is to consider the empirical distribution of the particles and to study its limiting behavior when the number of particles tends to infinity. Under suitable assumptions on the system it is often possible to describe the limit via a so-called \emph{McKean--Vlasov (MKV) equation}. The macroscopic behavior is also closely related to the so-called \emph{propagation of chaos property}, which roughly speaking means that an asymptotic i.i.d. property of the initial distributions propagates to later times.

\subsection{Contributions of the article}
In this paper we study two questions related to MKV limits of interacting stochastic partial differential equations (SPDEs). First, we consider an \(N\)-particle system \(X^{N, 1}, \dots, X^{N, N}\) given by the weakly interacting SPDEs
\[
d X^{N, i}_t = A X^{N, i}_t dt + \mu (t, X^{N, i}_t, \mathcal{X}^N_t) dt + \sigma (t, X^{N, i}_t, \mathcal{X}^N_t) d W^i_t,\quad i = 1, \dots, N,
\]
where
\[
\mathcal{X}^N_t \triangleq \frac{1}{N} \sum_{i = 1}^N \delta_{X^{N, i}_t}, \quad t \in \mathbb{R}_+, 
\]
are the empirical distributions, \(A\) is the generator of a \(C_0\)-semigroup \(S = (S_t)_{t \geq 0}\) and \(W^1, \dots, W^N\) are independent standard cylindrical Brownian motions. 
The natural candidate for a MKV limit of this particle system is the law of the MKV SPDE
\begin{align} \label{eq: MKV SPDE introduction}
d X_t = A X_t dt + \mu (t, X_t, P^X_t) dt + \sigma (t, X_t, P^X_t) d W_t,  
\end{align}
where \(P^X_t\) denotes the law of \(X_t\) and \(W\) is a standard cylindrical Brownian motion.
Under certain assumptions on the initial distribution, compactness of \(S\), and linear growth and continuity assumptions on the coefficients \(\mu\) and \(\sigma\), 
the first main contribution (\Cref{theo: POC new1}) of this paper is the following: 

If the MKV SPDE \eqref{eq: MKV SPDE introduction} satisfies uniqueness in law, then a unique law \(\mathcal{X}^0\) exists and \(\mathcal{X}^N\to \mathcal{X}^0\) in mean where \(\mathcal{X}^N\) and \(\mathcal{X}^0\) are considered as random variables with values in a Wasserstein space of probability measures on \(C([0, T], E)\), where \(T > 0\) is an arbitrary finite time horizon and \(E\) is the state space of the particles.
Furthermore, we provide a propagation of chaos result and we derive similar results (\Cref{theo: POC}) under Lipschitz conditions without a compactness assumption on the semigroup \(S\).

Besides interacting SPDEs we also investigate weak existence, pathwise uniqueness and uniqueness in law for the MKV SPDE as given in \eqref{eq: MKV SPDE introduction}.
More precisely, in \Cref{theo: main} we prove weak existence under a continuity (with the weak topology for the measure variable) and a linear growth assumption on \(\mu\) and \(\sigma\) and a compactness condition on the semigroup \(S\). 
For suitably integrable initial data, we replace the weak topology in the continuity assumption by a Wasserstein topology and we also relax the linear growth conditions, see \Cref{theo: main wasserstein cont}. Furthermore, in \Cref{theo: uni} we establish pathwise uniqueness and uniqueness in law under a modified Lipschitz condition.
Finally, in \Cref{theo: Lip} we provide an existence and uniqueness result under a classical Lipschitz condition which requires no additional assumptions on the semigroup~\(S\).

\subsection{Comments on related literature}
 For finite dimensional equations convergence to the MKV limit was systematically studied in \cite{https://doi.org/10.1002/mana.19881370116}.
For weakly interacting one-dimensional stochastic heat equations with Lipschitz coefficients which are linear in the measure variable, propagation of chaos was proved in \cite{M96}. The heat equation is included in our framework, see \Cref{ex: 1}.
Using our notation from above, convergence to the MKV limit for coefficients of the type
\begin{align}\label{eq: intro str coef}
\mu (t, x, \nu) \equiv \mu_1 (t, x) + \int \mu_2 (x, y) \nu(dy), \qquad \sigma (t, x, \nu) \equiv \sigma (t, x),
\end{align}
was proved in \cite[Theorem~5.3]{BKKX98} under assumptions on the initial distributions which are similar to ours (see Condition \ref{I} below and Eq. (3.10) in \cite{BKKX98}), certain assumptions on \(A\), see \Cref{rem: comp A} for comments, and linear growth and Lipschitz conditions on \(\mu_1, \mu_2\) and~\(\sigma\). Notice that the diffusion coefficient in \eqref{eq: intro str coef} is independent of the measure variable and that the drift coefficient depends linearly on it. In this paper we present results for more general coefficients \(A, \mu\) and \(\sigma\). In particular, in our main \Cref{theo: POC new1} we impose no Lipschitz conditions. For i.i.d. initial data the propagation of chaos result \cite[Theorem~5.3]{BKKX98} is covered by \Cref{theo: POC}, whose proof appears to us more straightforward. More precisely, it adapts the finite dimensional case (\cite{Lak18}).
Compared to \cite{BKKX98}, we establish stronger convergence results as we prove convergence in mean for random variables in a Wasserstein space, while convergence in probability and the weak topology are used in \cite{BKKX98}.
The basic structure of the proof for \cite[Theorem~5.3]{BKKX98} is similar to those of \Cref{theo: POC new1} in the sense that we also prove tightness and then use a martingale problem argument. The proof for tightness in~\cite{BKKX98} is an adaption of Kolmogorov's tightness criterion, while we use the compactness method from \cite{doi:10.1080/17442509408833868}. In addition, we prove tightness for random variables with values in a suitable Wasserstein space which is not done in \cite{BKKX98}.

Various existence and uniqueness results for MKV SPDEs were proved in \cite{AHMED199565,BKKX98,Gov15,HUANG2021112167,Ren21}.
With the exception of \cite{HUANG2021112167}, the conditions in these references are of Lipschitz type. \Cref{theo: main} is closely related to \cite[Theorem~2.1]{HUANG2021112167}, which is an existence result for MKV SPDEs with uniformly bounded continuous coefficients (where the weak topology is used for the measure variable). Compared to this theorem, we require less assumptions on the parameters \(A, \mu\) and \(\sigma\) (see \Cref{rem: comp A} for some comments). 
Furthermore, \Cref{theo: main wasserstein cont} extends \cite[Theorem 2.1]{HUANG2021112167} in the direction that it only requires a continuity assumption for a Wasserstein instead of the weak topology. Thanks to this extension our result covers for instance linear (in the measure variable) coefficients of the type
\[
\mu (t, x, \nu) \equiv \int \mu^\circ (t, x, y) \nu(dy)
\]
for unbounded \(\mu^\circ\).
Similar to those of \cite[Theorem 2.1]{HUANG2021112167}, the proof of \Cref{theo: main} relies on an approximation scheme and the compactness method from \cite{doi:10.1080/17442509408833868}. In contrast to \cite{HUANG2021112167}, we use a martingale problem argument to identify the limit and we establish moment estimates to reduce assumptions on \(\mu\) and \(\sigma\). The martingale problem argument is robust w.r.t. the linearity \(A\), while the argument in \cite{HUANG2021112167} uses some properties of \(A\). 
When compared to existence results for classical SPDEs, \Cref{theo: main} can be viewed as an extension of the main results from \cite{doi:10.1080/17442509408833868} to a McKean--Vlasov framework.
For finite dimensional MKV equations, general existence and uniqueness results were proved in~\cite{Funkai84}. 
Our uniqueness result extends a theorem from \cite{Funkai84} and we also adapt the basic proof strategy to our infinite dimensional setting.

\subsection{Structure of the article and comments on notation}

The article is structured as follows: In \Cref{sec: main} we introduce our setting and present existence and uniqueness results for MKV SPDEs. The convergence of the particle system to its MKV limit and the propagation of chaos property are discussed in \Cref{sec: chaos main}. In the remaining sections we present the proofs for our results.

Before we turn to the main body of this paper, let us also comment on notation and terminology. In general, we follow the seminal monograph of Da Prato and Zabczyk \cite{DePrato}. We also refer to this monograph for background information on stochastic integration in infinite dimensions. Further standard references on infinite dimensional stochastic analysis are the monographs~\cite{gawarecki2010stochastic,roeckner15}.

\begin{convention}
If not indicated otherwise, \(C > 0\) denotes a generic constant which is allowed to depend on all fixed parameters in the respective context. We also use the convention that \(C\) might change from line to line. 
\end{convention}

\section{Existence and Uniqueness of McKean--Vlasov SPDEs} \label{sec: main}
Let \(E = (E, \langle \cdot, \cdot \rangle_E, \|\cdot\|_E)\) and \(H = (H, \langle \cdot, \cdot \rangle_H, \|\cdot \|_H)\) be separable real Hilbert spaces, denote the space of linear bounded operators \(H \to E\) by \(L (H, E)\) and the space of Hilbert--Schmidt operators \(H \to E\) by \(L_2 (H, E)\). Let \(M_c(E)\) be the space of probability measures on \((E, \mathcal{B}(E))\) endowed with the weak topology, i.e. the topology of convergence in distribution.
For \(p \geq 1\) let \(M^p_w(E)\) be the set of all \(\nu \in M_c(E)\) such that 
\[
\|\nu \|_p \triangleq \Big( \int \|y\|^p_E \nu(dy)\Big)^{1/p} < \infty.
\]
We endow \(M^p_w (E)\) with the \(p\)-Wasserstein topology (\cite[Section 5.1]{Carmona18}), which turns \(M^p_w(E)\) into a Polish space. 
Next, we introduce a quadruple \((A, \mu, \sigma, \eta)\) of \emph{coefficients}.
\begin{enumerate}
	\item[(i)] Let \(A \colon D(A) \subseteq E \to E\) be the generator of a \(C_0\)-semigroup \(S = (S_t)_{t \geq 0}\) on \(E\) and denote its adjoint by \(A^* \colon D(A^*) \subseteq E \to E\).
	\item[(ii)]  Let \(\mu \colon \mathbb{R}_+ \times E \times M_c(E) \to E\) and \(\sigma \colon \mathbb{R}_+ \times E \times M_c(E) \to L (H, E)\) be Borel functions. To be precise, for \(\sigma\) we mean that for every \(h \in H\) the \(E\)-valued function \(\sigma h\) is Borel.
	\item[(iii)] Let \(\eta \in M_c(E)\).
\end{enumerate}
In the following we use the notation \(P^X_t \triangleq P \circ X^{-1}_t\) for \(t \in \mathbb{R}_+\).
\begin{definition}
	We call a triplet \((\B, W, X)\) a \emph{martingale solution} to the MKV SPDE with coefficients \((A, \mu, \sigma, \eta)\) if \(\B\) is a filtered probability space with right-continuous and complete filtration which supports a standard cylindrical Brownian motion \(W\) and a continuous \(E\)-valued adapted process \(X\) such that the following hold:
	\begin{enumerate}
\item[\textup{(i)}]	\(X_0 \sim \eta\), i.e. \(X_0\) has law \(\eta\).
\item[\textup{(ii)}]	 Almost surely for all \(t \in \mathbb{R}_+\)
	\begin{align*}
	\int_0^t \|S_{t - s} \mu(s, X_s, P^X_s)\|_E ds + \int_0^t \|S_{t - s} \sigma(s, X_s, P^X_s)\|_{L_2 (H, E)}^2 ds < \infty.
	\end{align*}
	\item[\textup{(iii)}]	Almost surely for all \(t \in \mathbb{R}_+\)
	\begin{equation*} 
	\begin{split}
	X_t = S_t X_0 + \int_0^t S_{t - s} \mu(s, X_s, P^X_s) ds + \int_0^t S_{t - s} \sigma(s, X_s, P^X_s) d W_s.
	\end{split}
	\end{equation*}
	\end{enumerate}
	We call \(X\) a \emph{solution process} and its law, seen as a Borel probability measure on \(C(\mathbb{R}_+, E)\) endowed with the local uniform topology, a \emph{solution measure}. The pair \((\B, W)\) is called a \emph{driving system}. Furthermore, for \(p \geq 1\) we call the solution measure a \emph{\(p\)-solution measure} if \(P^X_t \in M^p_w(E)\) for all \(t \in \mathbb{R}_+,\) and \((t \mapsto P^X_t) \in C(\mathbb{R}_+, M^{p}_w(E))\). In the same manner, we say that \(X\) is a \emph{\(p\)-solution process} if its law is a \(p\)-solution measure and in this case we call \((\B, W, X)\) a \emph{\(p\)-martingale solution}.
\end{definition}
\begin{remark}
	In case one is only interested in \(p\)-martingale solutions it suffices that \(\mu\) and \(\sigma\) are defined on \(\mathbb{R}_+ \times E \times M^p_w(E)\). Of course, in this case also the initial law \(\eta\) has to be taken from~\(M^p_w(E)\).
\end{remark}

From now on we fix \[0 < \alpha < 1/2 \quad \text{ and } \quad p' > 1/\alpha.\]
Let \((A, \mu, \sigma, \eta)\) be coefficients for a MKV SPDE. 
We formulate the following conditions:
\begin{enumerate}
	\item[\mylabel{A1}{\textup{(A1)}}] 
	\(A\) generates a compact \(C_0\)-semigroup \(S = (S_t)_{t \geq 0}\), i.e. \(S_t\) is compact for all \(t > 0\).
	\item[\mylabel{A2}{\textup{(A2)}}] 
For all \(y^* \in D(A^*)\) and \(t > 0\) the maps	\(\langle \mu (t, \cdot, \cdot), y^*\rangle_E\) and \(\| \sigma^*(t, \cdot, \cdot) y^*\|_H\) are continuous on \(E \times M_c (E)\). 
	\item[\mylabel{A3}{\textup{(A3)}}] 
For every \(T > 0\) there exists a Borel function \(\f = \f_T\colon (0, T] \to [0, \infty]\) and a constant \(C_T > 0\) such that 
\[
\int_0^T \Big[\frac{\f (s) }{s^{\alpha}} \Big]^2 ds < \infty,
\]
and 
\begin{align*}
\|S_t \sigma (s, x, \nu)\|_{L_2(H, E)}& \leq \f (t) \big(1 + \|x\|_E\big),
\end{align*}
and
\begin{align}
\label{eq: LG mit op norm}
\| \mu (s, x, \nu) \|_E + \|\sigma (s, x, \nu)\|_{L(H, E)} &\leq C_T \big(1 + \|x\|_E\big),
\end{align}
for all \(0 < t, s \leq T, x \in E\) and \(\nu \in M_c(E)\).
\end{enumerate}
\begin{remark}
In \ref{A3} the dependence of the function \(\f\) on a time horizon \(T > 0\) localizes the time variable of the coefficient \(\sigma\).
\end{remark}
	Condition \ref{A3} is closely connected to the factorization method of Da Prato, Kwapien, and Zabczyk~\cite{doi:10.1080/17442508708833480}, which we also use in our proofs. 
In the following remark we relate \ref{A3} to some conditions appearing in the literature on MKV SPDEs.
\begin{remark} \label{rem: comp A} \quad
	\begin{enumerate}
	\item[\textup{(i)}]
	In case \eqref{eq: LG mit op norm} holds, \ref{A3} is implied by
		\begin{align}\label{eq: fact cond semi}
		\int_0^t \frac{\|S_{s}\|_{L_2(E)}^2ds}{s^{2 \alpha}} < \infty, \quad \forall t > 0,
	\end{align}
	which is a classical condition appearing for instance in \cite{DePratoEd1,doi:10.1080/17442509408833868} for SPDEs without measure dependence.

Suppose that \(A\) is a negative definite self-adjoint operator\footnote{A negative definite self-adjoint operator generates a contraction semigroup, see \cite[Proposition 6.14]{Schmu12}.} and that there exists a \(\delta \in (0, 1)\) such that \((-A)^{-1 + \delta}\) is of trace class, i.e. 
	\begin{align}\label{eq: negative definite integr}
	\sum_{k = 1}^\infty \lambda_k^{-1 + \delta} < \infty,
	\end{align}
	where \(0 < \lambda_1 \leq \lambda_2 \leq \cdots\) are all eigenvalues of \(-A\), counting multiplicities, with \(- A e_k = \lambda_k e_k\) for an orthonormal basis \((e_k)_{k \in \mathbb{N}}\) of \(E\).
	This assumption appears as (a1) in \cite{HUANG2021112167}. 
	If it is in force, \eqref{eq: fact cond semi} holds
	 with \(\alpha = \delta/2\), since
	\begin{align*}
\int_0^t \frac{\|S_s\|^2_{L_2(E)} ds}{s^{\delta}}= \sum_{k = 1}^\infty \int_0^t \frac{e^{- 2\lambda_k s} ds}{s^{\delta}}
\leq  \int_0^\infty\frac{e^{-2z} dz}{z^{\delta}} \sum_{k = 1}^\infty  \lambda_k^{-1 + \delta} < \infty, \quad t > 0.
	\end{align*}
	
	In the paper \cite{BKKX98} the linearity \(A\) is also assumed to be negative definite and self-adjoint but only a weaker form of \eqref{eq: negative definite integr} is imposed. To give some details, for the diffusion coefficient \(\sigma (t, x, \nu) \equiv \sigma (t, x)\) it is assumed that there are non-negative numbers \(b_1, b_2, \dots\) such that
	\[
	\| \sigma^* (t, x) e_k \|_H^2 \leq b_k^2 \big(1 + \|x\|^2_E\big), \quad t \leq T, \ k = 1, 2, \dots,
	\]
	and 
	\[
	\sum_{k = 1}^\infty b^2_k\lambda_k^{-\theta} < \infty
	\]
	for some \(\theta \in (0, 1)\), cf. Eq. (2.44) in \cite{BKKX98}. 
	Under this condition the first part of \ref{A3} holds with \[\f (t) = \f_T(t) \triangleq C \Big(\sum_{k = 1}^\infty e^{- 2\lambda_k t} b_k^2\Big)^{1/2}, \quad t \in (0, T],\]
	and \(\alpha \triangleq (1 - \theta)/2 \in (0, 1/2)\), since
	\[
\|S_t \sigma (s, x)\|_{L_2 (H, E)}^2 = \sum_{k = 1}^\infty \|S_t \sigma(s, x) e_k\|^2_E \leq \sum_{k = 1}^\infty e^{- 2\lambda_k t} b_k^2 (1 + \|x\|_E^2),
\]
for \(0 < t, s \leq T, x \in E\), 
	and
	\[
	\int_0^t \Big[\frac{\f (s) }{s^{\alpha}} \Big]^2 ds = \sum_{k = 1}^\infty \int_0^t \frac{e^{- 2 \lambda_k s} b^2_k ds}{s^{1 - \theta}} \leq \int_0^\infty \frac{e^{- 2 z} dz}{z^{1 - \theta}} \sum_{k = 1}^\infty b^2_k \lambda_k^{- \theta} < \infty, \quad t \leq T.
	\]

	In a certain sense, \eqref{eq: fact cond semi} is close to optimal for the existence of solutions to Cauchy problems. More precisely, the stochastic Cauchy problem
	\[
	d X_t = A X_t dt + d W_t
	\]
	has a mild solution (with not necessarily continuous paths) if and only if 
	\begin{align} \label{eq: A1 with alpha = 0}
	\int_0^t \|S_t\|^2_{L_2(E)} dt < \infty, \quad t > 0,
	\end{align}
	see \cite[Theorem 7.1]{Ner05}.
	Moreover, if \(-A\) is self-adjoint, non-negative definite and there exists a complete orthonormal family of eigenvectors corresponding to its set of positive eigenvalues, then \eqref{eq: A1 with alpha = 0} is sufficient for the existence of a mild solution with continuous paths, see \cite[Theorem 1]{Iscoe90}.
	\item[\textup{(ii)}]
		Another classical type of linear growth condition is the following: For every \(T > 0\) there exists a constant \(C_T > 0\) such that 
	\[
	\| \mu (s, x, \nu) \|_E + \|\sigma (s, x, \nu)\|_{L_2(H, E)} \leq C_T \big(1 + \|x\|_E\big)
	\]
	for all \(s \leq T, x \in E\) and \(\nu \in M_c(E)\).
	The difference to \eqref{eq: LG mit op norm} is that this condition uses the Hilbert--Schmidt norm instead of the operator norm. Under this condition, \ref{A3} holds with \(\f (t) = \|S_t\|_{L(E)}\) for \(t > 0\), as there are constants \(M \geq 1\) and \(\omega \in \mathbb{R}_+\) such that \(\|S_t\|_{L (E)} \leq Me^{\omega t}\) for all \(t > 0\).
	\end{enumerate}
\end{remark}
	Our first main result is the following:
\begin{theorem} \label{theo: main}
	Suppose that \ref{A1}, \ref{A2} and \ref{A3} hold. Then, for every \(\eta \in M_c(E)\) there exists a martingale solution to the MKV SPDE with coefficients \((A, \mu, \sigma, \eta)\). 
\end{theorem}
The proof of \Cref{theo: main} is given in \Cref{sec: pf theo 1}. Let us explain two typical situations where the above theorem can be applied.

\begin{example}[\cite{doi:10.1080/17442509408833868}] \label{ex: 1}
Let \(\mathcal{O}\) be a bounded region in \(\mathbb{R}^d\) with smooth boundary and set \(E \triangleq L^2(\mathcal{O})\). If \(A\) is a strongly elliptic operator of order \(2m > d\) (with Dirichlet boundary conditions), then there exists an \(\alpha\) such that \eqref{eq: fact cond semi} holds, see \cite[Example 3]{doi:10.1080/17442509408833868}. Thus, by virtue of part (i) of \Cref{rem: comp A}, \Cref{theo: main} covers for instance the McKean--Vlasov stochastic heat equation 
\[
d X_t = \Delta X_t dt + \mu(t, X_t, P^X_t) dt + \sigma (t, X_t, P^X_t) d W_t
\]
with white noise \(W\) in case \(d = 1\).
\end{example}

\begin{example}
	Let \(E = L^2(\mathcal{O})\) be as in \Cref{ex: 1}. If \(A\) is a strongly elliptic second order operator (with Dirichlet boundary conditions), then \ref{A1} holds, see \cite[Remark 1]{doi:10.1080/17442509408833868} and \cite[Appendix A.5.2]{DePrato}.
	Thus, by virtue of part (ii) of \Cref{rem: comp A}, \Cref{theo: main} applies for instance to the McKean--Vlasov stochastic heat equation 
	\[
	d X_t = \Delta X_t dt + \mu(t, X_t, P^X_t) dt + \sigma (t, X_t, P^X_t) d B_t
	\]
	with colored noise \(B\) independent of the dimension \(d\) (recall that \(\mathcal{O} \subset \mathbb{R}^d\)).
\end{example}

Provided the initial value satisfies a suitable integrability condition, we can relax the continuity assumptions on \(\mu\) and \(\sigma\) in the measure variable by replacing the weak with a Wasserstein topology. Furthermore, we can strengthen the linear growth condition. Take \(1 \leq p^\circ < p'\).
	\begin{enumerate}
	\item[\mylabel{A4}{\textup{(A4)}}] 
	For all \(y^* \in D(A^*)\) and \(t > 0\) the maps	\(\langle \mu (t, \cdot, \cdot), y^*\rangle_E\) and \(\| \sigma^*(t, \cdot, \cdot) y^*\|_H\) are continuous on \(E \times M_w^{p^\circ} (E)\).
	\item[\mylabel{A5}{\textup{(A5)}}]
For every \(T > 0\) there exists a Borel function \(\f = \f_T \colon (0, T] \to [0, \infty]\) and a constant \(C_T > 0\) such that 
\[
\int_0^T \Big[\frac{\f (s) }{s^{\alpha}} \Big]^2 ds < \infty, 
\]
	and 
	\begin{align*}
	\|S_t \sigma (s, x, \nu)\|_{L_2(H, E)} &\leq \f (t) \big(1 + \|x\|_E + \|\nu\|_{p'}\big),
\\
	\|\mu (t, x, \nu)\|_E + \|\sigma(t, x, \nu)\|_{L(H,E)} &\leq C_T \big(1 + \|x\|_E + \|\nu\|_{p'}\big),
	\end{align*}
		for all \(0 < t, s \leq T, x \in E\) and \(\nu \in M^{p'}_w (E)\).
\end{enumerate}
\begin{theorem} \label{theo: main wasserstein cont}
Suppose that \(\eta \in M^{p'}_w(E)\) and that \(\mu\) and \(\sigma\) are only defined on \(\mathbb{R}_+ \times E \times M^{p^\circ}_w(E)\). 
Furthermore, assume that \ref{A1}, \ref{A4} and \ref{A5} hold.
Then, the MKV SPDE \((A, \mu, \sigma, \eta)\) has a \(p'\)-martingale solution.
\end{theorem}

	\Cref{theo: main wasserstein cont} can be proved similar to \Cref{theo: main} and we outline the few necessary changes in \Cref{sec: pf outline wasserstein}.

\begin{remark} \label{rem: weaker conti}
	Replacing \(M_c(E)\) by \(M_w^{p^\circ}(E)\) for the continuity assumptions on the coefficients \(\mu\) and \(\sigma\) is a useful generalization. For instance, consider the coefficient
	\[
	\mu (t, x, \nu) \equiv \int \mu^* (s, x, y) \nu(dy)
	\]
	for a measurable function \(\mu^*\). This coefficient is well-defined for all \(\nu \in M_c(E)\) only if \(y \mapsto \mu^*(t, x,y)\) is bounded. However, if we restrict our attention to \(\nu \in M^{p}_w(E)\) for some \(p \geq 1\), we can allow unbounded \(\mu^*\) under a suitable growth assumption on \(y \mapsto \mu^*(t, x, y)\).
\end{remark}

Next, we also provide a uniqueness result for MKV SPDEs. We fix~\(p \geq 2\) and we assume that \(\mu\) and \(\sigma\) are defined on \(\mathbb{R}_+ \times E \times M^p_w(E)\).
\begin{definition}
	Let \(\eta \in M^p_w(E)\).
	\begin{enumerate}
		\item[\textup{(i)}] We say that the MKV SPDE \((A, \mu, \sigma, \eta)\) satisfies \emph{\(p\)-uniqueness in law} if there is at most one \(p\)-solution measure.
		\item[\textup{(ii)}] We say that the MKV SPDE \((A, \mu, \sigma, \eta)\) satisfies \emph{\(p\)-pathwise uniqueness} if for any two \(p\)-martingale solutions \((\B, W, X)\) and \((\B, W, Y)\) we have a.s. \(X = Y\).
	\end{enumerate}
\end{definition}
Let \(\mathsf{w}_p\) be the \(p\)-Wasserstein metric, i.e. for \(\nu, \eta \in M^p_w (E)\) set
\[
\mathsf{w}_p (\nu, \eta) \triangleq \inf_{F \in \Pi (\nu, \eta)} \Big(\int \|x - y\|^p_E F(dx, dy)\Big)^{1/p},
\]
where \(\Pi (\nu, \eta)\) is the set of Borel probability measures \(F\) on \(E \times E\) such that \(F(dx \times E) = \nu(dx)\) and \(F(E \times dx) = \eta(dx)\). 
\begin{enumerate}
	\item[\mylabel{U1}{\textup{(U1)}}] For every \(T, m > 0\) there exist two Borel functions \(\f = \f_{T, m} \colon (0, T] \to [0, \infty]\) and \(\g = \g_{T, m} \colon (0, T] \to [0, \infty]\) such that 
	\[
	\int_0^T \Big(\Big[\frac{\f (s) }{s^{\alpha}} \Big]^2 + \big[ \g (s) \big]^{p / (p - 1)} \Big) ds < \infty, 
	\] 
	and an increasing continuous function \(\kappa = \kappa_{T, m} \colon \mathbb{R}_+ \to \mathbb{R}_+\) with \(\kappa (0) = 0, k(x) > 0\) for \(x > 0\), and 
	\[
	\int_{0+} \frac{dx}{|\kappa(x^{1/p})|^p} = \infty, 
	\]
	such that 
	\begin{align*}
	 \|S_t(\sigma (s, x, \nu) - \sigma (t, y, \eta))\|_{L_2(H, E)} &\leq \f (t) \big( \|x - y\|_E + \kappa (\mathsf{w}_p(\nu, \eta))\big),
	 \\
	 \|S_t (\mu (s, x, \nu) - \mu (t, y, \eta))\|_E &\leq \g (t) \big( \|x - y\|_E + \kappa (\mathsf{w}_p(\nu, \eta))\big),
	\end{align*}
	and
	\begin{align*}
	\|S_t \sigma (s, x, \nu)\|_{L_2 (H, E)} &\leq \f (t) \big( 1 + \|x\|_E\big),\\
		\|S_t \mu (s, x, \nu)\|_E &\leq \g (t) \big( 1 + \|x\|_E\big),
	\end{align*}
	for all \(0 < t, s \leq T, x, y\in E\) and \(\nu, \eta \in M^p_w(E)\) with \(\|\nu\|_p, \|\eta\|_p \leq m\). 
\end{enumerate}
Irrespective of \(p\), the identity \(\kappa (x) = x\) is a possible choice for \(\kappa\) and hence \ref{U1} can be seen as a generalized Lipschitz condition.

Below we use \ref{U1} together with the condition that \(p > 1/\alpha\), which excludes the case \(p = 2\). The following condition includes the case \(p = 2\).

\begin{enumerate}
	\item[\mylabel{U2}{\textup{(U2)}}] For every \(T, m > 0\) there exists a constant \(C = C_{T, m} > 0\) and an increasing continuous function \(\kappa = \kappa_{T, m} \colon \mathbb{R}_+ \to \mathbb{R}_+\) with \(\kappa (0) = 0, k(x) > 0\) for \(x > 0\), and 
	\[
	\int_{0+} \frac{dx}{|\kappa(x^{1/p})|^p} = \infty, 
	\]
	such that 
	\begin{align*}
	\|\mu (s, x, \nu) - \mu (t, y, \eta)\|_E + \|\sigma (s, x, \nu) &- \sigma (t, y, \eta)\|_{L_2(H, E)} \\&\leq C \big( \|x - y\|_E + \kappa (\mathsf{w}_p(\nu, \eta))\big),
	\end{align*}
	and
	\[
	\|\mu (s, x, \nu)\|_E + \|\sigma (s, x, \nu)\|_{L_2 (H, E)} \leq C \big( 1 + \|x\|_E\big),
	\]
	for all \(0 < t, s \leq T, x, y\in E\) and \(\nu, \eta \in M^p_w(E)\) with \(\|\nu\|_p, \|\eta\|_p \leq m\). 
\end{enumerate}

 Our next main result is the following:
 \begin{theorem} \label{theo: uni}
 	Suppose that either  \(p > 1/\alpha\) and 
 	\ref{U1} hold, or that \(p \geq 2\) and \ref{U2} hold.
 Then, for every \(\eta \in M^p_w(E)\) the MKV SPDE \((A, \mu, \sigma, \eta)\) satisfies \(p\)-uniqueness in law and \(p\)-pathwise uniqueness.
 \end{theorem}

A different uniqueness result was established in \cite{HUANG2021112167}.
We prove \Cref{theo: uni} in \Cref{sec: pf theo uni}. 
The  \Cref{theo: main wasserstein cont,theo: uni} can be combined to an existence and uniqueness statement. However, for the existence part we always require that \(A\) generates a compact semigroup. We now also provide a more classical existence and uniqueness result for equations with Lipschitz coefficients which needs no compactness assumption.
 \begin{enumerate}
 	\item[\mylabel{L1}{\textup{(L1)}}]
 	For every \(T > 0\) there exists two Borel functions \(\f = \f_T \colon (0, T] \to [0, \infty]\) and \(\g = \g_T \colon (0, T] \to [0, \infty]\) such that 
 	\[
 	\int_0^T \Big(\Big[\frac{\f (s) }{s^{\alpha}} \Big]^2  + \big[ \g (s) \big]^{p / (p - 1)} \Big) ds < \infty, 
 	\] 
 	and 
 		\begin{align*}
  \|S_t(\sigma (s, x, \nu) - \sigma (s, y, \eta))\|_{L_2(H, E)} &\leq \f (t) \big( \|x - y\|_E + \mathsf{w}_p(\nu, \eta)\big), \\
  	\|S_t (\mu (s, x, \nu) - \mu (s, y, \eta))\|_E &\leq \g (t) \big( \|x - y\|_E + \mathsf{w}_p(\nu, \eta)\big),
 	\end{align*}
 	and 
 	\begin{align*}
 	 \|S_t \sigma (s, x, \nu)\|_{L_2 (H, E)} &\leq \f (t) \big( 1 + \|x\|_E + \|\nu\|_p\big),\\
 	 \|S_t \mu (s, x, \nu)\|_E &\leq \g (t) \big( 1 + \|x\|_E + \|\nu\|_p\big),\
 	\end{align*}
 	for all \(0 < t, s \leq T,\) \(x, y\in E\) and \(\nu, \eta \in M^p_w(E)\).
 \end{enumerate}
Below we use \ref{L1} together with the condition that \(p > 1/\alpha\), which excludes the case \(p = 2\). The following Lipschitz condition includes the case \(p = 2\).
\begin{enumerate}
	\item[\mylabel{L2}{\textup{(L2)}}]
	For every \(T> 0\) there exists a constant \(C = C_T > 0\) such that
	\begin{align*}
	\|\mu (t, x, \nu) - \mu (t, y, \eta)\|_E + \|\sigma (t, x, \nu) &- \sigma (t, y, \eta)\|_{L_2(H, E)} \\&\leq C \big( \|x - y\|_E + \mathsf{w}_p(\nu, \eta)\big),
	\end{align*}
	 for all \(0 < t \leq T,\) \(x, y\in E\) and \(\nu, \eta \in M^p_w(E)\).
	Moreover, the functions \(\|\mu (\cdot, 0, \delta_0)\|_E\) and \(\|\sigma (\cdot, 0, \delta_0)\|_{L_2(H, E)}\) are bounded on compact subsets of \(\mathbb{R}_+\).
\end{enumerate}
Recall that \(S\) is called a \emph{generalized contraction semigroup} if there exists an \(\omega \in \mathbb{R}\) such that \(\|S_t\|_{L(E)} \leq e^{\omega t}\) for all \(t \in \mathbb{R}_+\). 

\begin{theorem} \label{theo: Lip}
	Suppose that \(\eta \in M^p_w (E)\) and that either \(p > 1/\alpha\) and \ref{L1} hold, or that \(p \geq 2\), that \ref{L2} holds and that \(S\) is a generalized contraction. Then, the MKV SPDE \((A, \mu, \sigma, \eta)\) has a \(p\)-martingale solution and it satisfies \(p\)-uniqueness in law and \(p\)-pathwise uniqueness. Moreover, the MKV SPDE can be realized on any driving system~\((\B, W)\).
\end{theorem}

\Cref{theo: Lip} can be seen as a version of \cite[Theorem 4.21]{Carmona18} for an infinite dimensional setting and its proof is similar. 
For completeness we provide it in \Cref{sec: pf theo Lip}.

For negative definite self-adjoint \(A\) satisfying a (generalized) variant of \eqref{eq: fact cond semi}, a related existence and uniqueness result is given by \cite[Theorem 3.1]{Ren21}.
Notice that any semigroup with negative self-adjoint generator is a contraction (\cite[Proposition~6.14]{Schmu12}). 

 \section{Propagation of Chaos for Weakly Interacting SPDEs} \label{sec: chaos main}
  In this section we discuss the chaotic property of weakly interacting particles which are modeled as mild solutions to SPDEs. The section is split into two parts. In the first we derive a result under a continuity condition on the coefficients and a uniqueness assumption on the law of the limiting MKV SPDE. In the second part we provide a result under Lipschitz conditions on the coefficients.
  
 \subsection{The Chaotic Property under Continuity and Uniqueness Assumptions} \label{sec: POC new}
   To fix our setting, we assume that \(\alpha, p', A, \mu\) and \(\sigma\) are as in \Cref{sec: main} with the important exception that \(\mu\) and \(\sigma\) are only defined on \(\mathbb{R}_+ \times E \times M^{p^\circ}_w(E)\) for some \(1 \leq p^\circ < p'\). 
 For \(N \in \mathbb{N}\) we define a map \(L^N \colon E^{\otimes N} \to M_c(E)\) by 
  	\[
 L^N (x_1, \dots, x_N) \triangleq \frac{1}{N} \sum_{i = 1}^N\delta_{x_i}, \quad (x_1, \dots, x_N) \in E^{\otimes N}.
 \]

Let us start with a condition for the initial laws:
\begin{enumerate}
\item[\mylabel{I}{\textup{(I)}}] Let \(\eta^N \in M_c (E^{\otimes N})\) be symmetric such that there exists a measure \(\eta \in M_c(E)\) with
\begin{align}\label{eq: symmetry initial laws assump}
\eta^N (L^N \in d x) \xrightarrow{\quad} \delta_{\eta} (dx) 
\end{align}
weakly\footnote{Notice that the probability measures in \eqref{eq: symmetry initial laws assump} are elements of \(M_c (M_c(E))\) and that the weak topology refers to the topology of convergences in distribution on the space \(M_c(M_c(E))\).} as \(N \to \infty\).
Moreover,
\[
\sup_{N \in \mathbb{N}} \int \|\mathsf{X}_1^N (x)\|^{p'}_E \eta^N(dx) < \infty, 
\]
where \(\X^N_1 \colon E^{\otimes N}\to E\) denotes the projection to the first coordinate.
\end{enumerate}
\begin{remark}
	According to \cite[Proposition I.2.2]{10.1007/BFb0085169}, \eqref{eq: symmetry initial laws assump} holds if and only if the sequence \(\eta^1, \eta^2, \dots\) is \emph{\(\eta\)-chaotic}, i.e. for every \(k \in \mathbb{N}\) and \(\phi_1, \phi_2, \dots, \phi_k \in C_b(E)\)
	\[
	\int_{E^{\otimes N}} \prod_{i = 1}^k \phi_i (\X^N_i (x)) \eta^N(dx) \to \prod_{i = 1}^k \int_E \phi_i (x) \eta (dx), \quad N \to \infty.
	\]
	Moreover, it is also equivalent to the above for \(k = 2\). 
\end{remark}

 The following condition deals with the existence of weakly interacting particles whose chaotic behavior we investigate in the remainder of this section.
\begin{enumerate} 
 	\item[\mylabel{E}{\textup{(E)}}] For \(N \in \mathbb{N}\) there exists a filtered probability space \(\B^N\) with right-continuous and complete filtration which supports independent standard cylindrical Brownian motions \(W^{1} \equiv W^{N, 1}, \dots, W^{N} \equiv W^{N, N}\) and mild solution processes \(X^{N, 1}, \dots, X^{N, N}\) to the~SPDE
 	\[
 	d X^{N, i}_t = A X^{N, i}_t dt + \mu (t, X^{N, i}_t, \mathcal{X}^N_t) dt + \sigma (t, X^{N, i}_t, \mathcal{X}^N_t) d W^{i}_t, 
 	\]
 	with \((X^{N, 1}_0, \dots, X^{N, N}_0) \sim \eta^N\) and 
 	\[
 	\mathcal{X}^N_t \triangleq \frac{1}{N} \sum_{i = 1}^N \delta_{X^{N, i}_t} = L^N (X^{N, 1}_t, \dots, X^{N, N}_t), \quad t \in \mathbb{R}_+.
 	\]
 	Moreover, the SPDE
 	\[
 	d Y_t = \Big[\bigoplus_{i = 1}^N A\Big] Y_t dt + \Big[\bigoplus_{i = 1}^N \mu (t, Y_t, L^N(Y_t))\Big] dt + \Big[\bigoplus_{i = 1}^N \sigma (t, Y_t, L^N (Y_t))\Big] d W_t, 
 	\]
 	satisfies uniqueness in law.
 \end{enumerate}

Next, we formulate a uniqueness condition for the limiting MKV SPDE.
 \begin{enumerate}
 	\item[\mylabel{UL}{\textup{(UL)}}] The MKV SPDE with coefficients \((A, \mu, \sigma, \eta)\) satisfies \(p'\)-uniqueness in law. 
\end{enumerate}

Finally, we also formulate a version of \ref{A4} from \Cref{sec: main}.
\begin{enumerate}
	\item[\mylabel{C}{\textup{(C)}}] 
	For all \(y^* \in D(A^*)\) and \(t > 0\) the maps \(\langle \mu (t, \cdot, \cdot), y^*\rangle_E\) and \(\| \sigma^*(t, \cdot, \cdot) y^*\|_H\) are continuous on \(E \times M_w^{p^\circ} (E)\), and the maps  \(\langle \mu, y^*\rangle_E\) and \(\| \sigma^*y^*\|_H\) are bounded on compact subsets of~\(\mathbb{R}_+ \times E \times M^{p^\circ}_w (E)\).
 \end{enumerate}
\begin{remark}
The final local boundedness condition from \ref{C} is not implied by the linear growth condition \ref{A5}, because \(M^{p'}_w (E) \subset M^{p^\circ}_w(E)\) since \(p^\circ < p'\).
\end{remark}
 
 For \(T > 0\) let \(\mathsf{w}^{p^\circ}_T\) be the \(p^\circ\)-Wasserstein metric on \(M^{p^\circ}_w(C([0, T], E))\) where \(C([0, T], E)\) is endowed with the uniform metric.
 The following theorem is the main result in this section. It formalizes the chaotic behavior of the weakly interacting SPDEs from \ref{E}.
 
 \begin{theorem} \label{theo: POC new1}
 	Suppose that \ref{A1}, \ref{A5}, \ref{I}, \ref{E}, \ref{UL} and \ref{C}  hold.
 	Then, the MKV SPDE with coefficients \((A, \mu, \sigma, \eta)\) has a \(p'\)-martingale solution with unique law \(\mathcal{X}^0\).
 	Moreover, for all \(T > 0\)
 	\begin{align}\label{eq: new WS metric conv}
 	\lim_{N\to \infty} E \Big[ \big|\mathsf{w}^{p^\circ}_T (\mathcal{X}^N, \mathcal{X}^0)\big|^{p^\circ} \Big] = 0,
 	\end{align}
 	and the particles \(X^{N, i}\) are \(\mathcal{X}^0\)-chaotic, i.e. for every \(k \in \mathbb{N}\) we have 
	\begin{align} \label{eq: second conv}
(X^{N, 1}, \dots, X^{N, k}) \to (Y^1, \dots, Y^k)
\end{align}
 	weakly as \(N \to \infty\), where \(Y^1, \dots, Y^k\) are i.i.d. with \(Y^1 \sim \mathcal{X}^0\).
 \end{theorem}

The proof of \Cref{theo: POC new1} is given in \Cref{sec: pf POC new 1}. 
\Cref{eq: second conv} means that the particles \(X^{N, i}\) become asymptotically i.i.d. as \(N \to \infty\). 
 In the proof of \Cref{theo: POC new1} we establish the existence part without evoking results from \Cref{sec: main}. 
\Cref{theo: uni} provides some conditions for~\ref{UL}. More conditions for \ref{UL} can be found in \cite{HUANG2021112167}.

 \subsection{The Chaotic Property under Lipschitz Conditions} \label{sec:POC}
 
 In this section we discuss the chaotic behavior for weakly interacting SPDEs with Lipschitz coefficients. 
 
 Let \(A, \mu\) and \(\sigma\) be as in \Cref{sec: main} but \(\mu\) and \(\sigma\) need only be defined on \(\mathbb{R}_+ \times E \times M^p_w(E)\).
 Take a filtered probability space \(\B = (\Omega, \cF, (\cF_t)_{t \geq 0}, P)\) which supports a sequence \(W^1, W^2, \dots\) of independent standard cylindrical Brownian motions and a sequence \(\xi^1_0, \xi^2_0, \dots\) of \(\cF_0\)-measurable i.i.d. random variables with \(\xi_0^1 \sim \eta \in M^p_w (E)\). 
The following proposition shows that \ref{E} is implied by the global Lipschitz condition \ref{L1}.
\begin{proposition} \label{prop: weakly interact existence}
	Assume that \ref{L1} holds. 
	For \(N \in \mathbb{N}\) and \(i = 1, \dots, N\), on \(\B\) there exists a unique (up to indistinguishability) mild solution process \(X^{i, N}\) to the SPDE
 \[
 d X^{N, i}_t = A X^{N, i}_t dt + \mu (t, X^{N, i}_t, \mathcal{X}^N_t) dt + \sigma (t, X^{N, i}_t, \mathcal{X}^N_t) d W^{i}_t, \quad X^{N, i} = \xi^i_0, \]
 with 
 \[
 \mathcal{X}^N_t \triangleq \frac{1}{N} \sum_{i = 1}^N \delta_{X^{N, i}_t}, \quad t \in \mathbb{R}_+.
 \]
\end{proposition}
\Cref{prop: weakly interact existence} follows from \Cref{theo: DZ}. For completeness we give a proof in \Cref{sec: pf prop ex}.
The following theorem is a version of \Cref{theo: POC new1} for the present setting. Compared to \Cref{theo: POC new1} its scope is slightly different as the semigroup \(S\) needs not to be compact but the coefficients have to be Lipschitz. 

For \(T > 0\) recall that \(\mathsf{w}^p_T\) is the \(p\)-Wasserstein metric on \(M^p_w(C([0, T], E))\) where \(C([0, T], E)\) is endowed with the uniform metric.
In case \(p > 1/\alpha\) and \ref{L1} hold, or \(p \geq 2\), \ref{L2} holds and that \(S\) is a generalized contraction, \Cref{theo: Lip} implies the existence of a \(p\)-martingale solution to the MKV SPDE \((A, \mu, \sigma, \eta)\) with a unique law \(\mathcal{X}^0\).
\begin{theorem} \label{theo: POC}
		Assume that either \(p > 1/\alpha\) and that \ref{L1} hold, or that \(p \geq 2\), \ref{L2} holds and that \(S\) is a generalized contraction. 
	For every \(T > 0\) it holds that 
	\begin{align} \label{eq: first conv}
	\lim_{N \to \infty} E \Big[ \big|\mathsf{w}_T^p (\mathcal{X}^N, \mathcal{X}^0)\big|^p \Big] = 0.
	\end{align}
	Moreover, the particles \(X^{N, i}\) are \(\mathcal{X}^0\)-chaotic, see \Cref{theo: POC new1}.
\end{theorem}
\Cref{theo: POC} can be proved as its finite dimensional counterpart \cite[Theorem 3.3]{Lak18}. For completeness we give a proof in \Cref{sec: pf POC lip}. 
Except of our assumption that we use i.i.d. initial data, \Cref{theo: POC} generalizes \cite[Theorem 5.3]{BKKX98} to more general particle systems. In particular, the convergence in \Cref{theo: POC} is stronger and its proof appears to be more straightforward.

\section{Proof of \Cref{theo: main}} \label{sec: pf theo 1}
The proof is split into several steps. In Step 0 we prepare some estimates and we recall the factorization formula from \cite{doi:10.1080/17442508708833480}. Then, in Step 1
we define an approximation sequence, in Step 2 we establish some moment estimates, in Step 3 we verify tightness of the approximation sequence and in Step 4 we investigate a martingale problem. In the fifth and final step we use a representation theorem for cylindrical continuous local martingales to complete the proof.

~

\noindent
\emph{Step 0: A short Recap of the Factorization Formula.} 
Fix a finite time horizon \(T > 0\).
For \(p' > 1,\) \(1/p' < \lambda \leq 1\) and \(h \in L^{p'}([0, T], E)\) we set 
\[
R_\lambda h (t) \triangleq \int_0^t (t - s)^{\lambda - 1} S_{t-s} h (s) ds, \quad t \leq T.
\]
Note that \(R_\lambda\) is indeed well-defined, as
\begin{equation} \begin{split} \label{eq: fak bound 1}
\int_0^t (t - s)^{\lambda - 1} &\|S_{t-s} h (s)\|_E ds 
\\&\leq \Big(\int_0^T s^{p' (\lambda - 1)/(p' - 1)} \|S_{s}\|^{p'/(p' - 1)}_{L(E)} ds\Big)^{(p' - 1)/p'} \Big( \int_0^T \|h(s)\|^{p'}_E ds \Big)^{1/p'},
\end{split}
\end{equation}
by H\"older's inequality. The first integral is finite as \(p'(\lambda - 1)/(p' - 1) > - 1\Longleftrightarrow\lambda > 1/p'\). 
The inequality \eqref{eq: fak bound 1} shows that 
\[
\|R_\lambda h (t)\|_E \leq \Big(\int_0^T s^{p' (\lambda - 1)/(p' - 1)} \|S_{s}\|^{p'/(p' - 1)}_{L(E)} ds\Big)^{(p' - 1)/p'} \Big( \int_0^T \|h(s)\|^{p'}_E ds \Big)^{1/p'},
\]
which means that \(R_\lambda\) is a bounded linear operator on \(L^{p'}([0, T], E)\).
\begin{lemma}[\textup{\cite[Proposition 1]{doi:10.1080/17442509408833868}}] \label{lem: R compact}
	For any \(1/p' < \lambda \leq 1\) the operator \(R_\lambda\) maps \(L^{p'}([0, T], E)\) into \(C([0, T], E)\). Moreover, if the semigroup \(S\) is compact, then \(R_\lambda\) is compact.
\end{lemma}

Next, take some \(0 < \alpha < 1/2\) and \(p' > 2\) large enough such that \(1/p' < \alpha\). Moreover, let \(\f \colon (0, T] \to [0, \infty]\) be a Borel function such that 
\[
\int_0^T \Big[\frac{\f (s) }{s^{\alpha}} \Big]^2 ds < \infty,
\]
and let \(\phi\) be a predictable \(L (H,E)\)-valued process and \(\psi\) a predictable real-valued process such that
\[
\|S_t \phi_s \|_{L_2 (H, E)} \leq \f (t) | \psi_s |
\]
for all \(0 < t, s \leq T\), and 
\[
E \Big[ \int_0^T |\psi_s|^{p'} ds \Big] < \infty.
\]
Set \(\gamma \triangleq p'/(p' - 1)\). Then, using H\"older's inequality in the second and last line, and Young's inequality in the third line, we obtain
\begin{align*}
&\hspace{-0.3cm}\int_0^T (T - t)^{\alpha - 1} \Big( E \Big[ \int_0^t (t - s)^{- 2\alpha} \|S_{t - s} \phi_s\|^2_{L_2(H, E)} ds \Big]\Big)^{1/2} dt 
\\&\leq \Big(\int_0^T s^{\gamma (\alpha - 1)} ds \Big)^{1/\gamma} \Big(\int_0^T \Big( \int_0^T \Big[\frac{\f (t - s) }{(t - s)^{\alpha}} \Big]^2 E\big[ |\psi_s|^2 \big] ds \Big)^{p'/2} dt \Big)^{1/p'}
\\&\leq \Big(\int_0^T s^{\gamma (\alpha - 1)} ds \Big)^{1/\gamma} \Big( \int_0^T \Big[\frac{\f (s) }{s^{\alpha}} \Big]^2 ds \Big)^{1/2} \Big( \int_0^T E \big[ |\psi_s|^2\big]^{p'/2} ds\Big)^{1/p'}
\\&\leq \Big(\int_0^T s^{\gamma (\alpha - 1)} ds \Big)^{1/\gamma} \Big( \int_0^T \Big[\frac{\f (s) }{s^{\alpha}} \Big]^2 ds \Big)^{1/2} \Big(E \Big[ \int_0^T |\psi_s|^{p'}ds\Big]\Big)^{1/p'}.
\end{align*}
Since
\(
\gamma (\alpha - 1) > -1 \ \Longleftrightarrow\ \alpha > \frac{1}{p'}, 
\)
the term in the last line is finite and the factorization formula \cite[Theorem~5.10]{DePrato} yields that 
\begin{align} \label{eq: fac fub}
\int_0^t S_{t-s} \phi_s d W_s = \frac{\sin (\pi \alpha)}{\pi} R_\alpha Y (t), \quad t \leq T,
\end{align}
with
\[
Y_t \triangleq \int_0^t (t - s)^{- \alpha} S_{t-s} \phi_s d W_s,
\]
where \(W\) is a standard cylindrical Brownian motion. In this formula the process \(Y\) has to be understood in the sense of the stochastic Fubini theorem (\cite[Theorem~4.33]{DePrato} or \cite[Proposition~6.1]{MR2067962}). In particular, latter yields that the stochastic convolution \(\int_0^t S_{t- s} \phi_s dW_s\) is well-defined. 

At the beginning of this step we defined \(R_\alpha\) on \(L^{p'}([0, T], E)\). We now show that a.a. paths of \(Y\) are in \(L^{p'}([0, T], E)\). Then, we can also conclude from Lemma \ref{lem: R compact} that the stochastic convolution has a continuous version.
We estimate
\begin{equation} \begin{split} \label{eq: fak bound 2}
E \Big[ \int_0^T \|Y_t\|^{p'}_E dt \Big] &\leq c_{p'} \int_0^T E \Big[ \Big(\int_0^t (t - s)^{- 2 \alpha} \|S_{t - s} \phi_s\|_{L_2(H, E)}^2 ds \Big)^{p'/2} \Big] dt
\\&\leq c_{p'}  E \Big[\int_0^T \Big(\int_0^t \Big[\frac{\f (t - s) }{(t - s)^{\alpha}} \Big]^2 |\psi_s|^2 ds \Big)^{p'/2} dt  \Big]
\\&\leq c_{p'} \Big( \int_0^T \Big[\frac{\f (s) }{s^{\alpha}} \Big]^2 ds \Big)^{p'/2}  E\Big[ \int_0^T |\psi_s|^{p'} ds \Big],
\end{split}
\end{equation}
where we use Burkholder's inequality (with constant \(c_{p'}\)) in the first and Young's inequality in the last line.
We conclude that the stochastic convolution \(\int_0^\cdot S_{\cdot - s} \phi_s dW_s\) has a continuous version.
Let us summarize the above observations.
\begin{lemma} \label{lem: factorization}
	Suppose that \(\alpha, p', \f, \phi\) and \(\psi\) are as above. 
	Then, the stochastic convolution 
	\[
	t \mapsto \int_0^t S_{t - s} \phi_s d W_s
	\]
	is well-defined and continuous. Furthermore, there exists a constant \(C\) depending on \(p', T\) and \(\f\) such that for every \(t \leq T\) 
	\begin{align}\label{eq: Burkholder fak}
	E \Big[ \sup_{s \leq t}\Big\| \int_0^s S_{s - r} \phi_r d W_r \Big\|^{p'}_E \Big] \leq C E \Big[ \int_0^t |\psi_s|^{p'} ds \Big].
	\end{align}
\end{lemma}
\begin{proof}
	The final estimate follows from \eqref{eq: fak bound 1} and \eqref{eq: fak bound 2}. All other claims were proved before.
\end{proof}
It seems that there is no estimate of the type \eqref{eq: Burkholder fak} in the monograph \cite{DePrato}. However, a related one can be found in its first edition, namely \cite[Proposition 7.9]{DePratoEd1}. 
\\

\noindent
\emph{Step 1: Definition of the Approximation Sequence.} 
Let \(0 < \alpha < 1/2\) and \(p' > 1/\alpha\) be as in \Cref{sec: main}.
Let \((\Omega, \mathcal{F}, (\cF_t)_{t \geq 0}, P)\) be a filtered probability space (with right-continuous and complete filtration) which supports a standard cylindrical Brownian motion \(W\) and an \(\mathcal{F}_0\)-measurable random variable \(\xi_0\) with distribution \(\eta\).
Take \(n \in \mathbb{N}\) and define a process \(X^n\) as follows: 
\(
X^n_0 \triangleq \xi_0 \1_{\|\xi_0\|_E \leq n}
\)
and for \(k\in \mathbb{Z}_+\) and \(k 2^{-n} < t \leq (k + 1) 2^{-n}\) we define inductively
\begin{align*}
X^n_t \triangleq S_{t - k 2^{-n}} X^n_{k 2^{-n}} &+ \int_{k 2^{-n}}^t S_{t - s} \mu(s, X^n_{k 2^{-n}}, P^{X^n}_{k 2^{-n}}) ds \\&+ \int_{k 2^{-n}}^t S_{t-s} \sigma  (s, X^n_{k 2^{-n}}, P^{X^n}_{k 2^{-n}}) d W_s, 
\end{align*} 
and 
\[
\nu_t^n \triangleq P^{X^n}_{k 2^{-n}}, \quad 
\mu^n (t, \omega, \nu) \triangleq \mu (t, \omega (k 2^{-n}), \nu), \quad \sigma^n (t, \omega, \nu) \triangleq \sigma (t, \omega (k 2^{-n}), \nu)
\]
where \((\omega, \nu) \in C(\mathbb{R}_+, E) \times M_c(E)\). At this point, recall the notation \(P^X_t = P \circ X_t^{-1}\).

Let us explain the induction procedure in more detail: 
Suppose that \(k \in \mathbb{Z}_+\) is such that \(X^n\) it is well-defined on \([0, k 2^{-n}]\) and 
\[
E \Big[ \sup_{s \leq k 2^{-n}} \|X_s^n\|^{p'}_E\Big] < \infty.
\]
Then, \Cref{lem: factorization} and the linear growth assumption \ref{A3} yield that \(X^n\) is also well-defined on \([0, (k + 1) 2^{-n}]\) and we also have 
\[
E \Big[\sup_{s \leq (k + 1) 2^{-n}} \|X_s^n\|^{p'}_E \Big] < \infty.
\]
The construction based on the factorization method yields that \(X^n\) has continuous paths. The following lemma collects our observations and further provides the dynamics of \(X^n\).

\begin{lemma} \label{lem: dynamics}
	The process \(X^n\) has a.s. continuous paths, for all \(T > 0\) it holds that 
	\[
	E\Big[\sup_{s \leq T} \|X^n_s\|^{p'} \Big] < \infty,
	\]
	and the dynamics of \(X^n\) are given by
	\[
	X^n_t = S_t X^n_0 + \int_0^t S_{t-s} \mu^n (s, X^n, \nu_s^n) ds + \int_0^t S_{t-s} \sigma^n (s, X^n, \nu_s^n) d W_s, \quad t \in \mathbb{R}_+. 
	\]
\end{lemma}
\begin{proof}
	It is only left to prove the formula for the dynamics. We use induction. Suppose that \(X^n\) has the claimed dynamics on \([0, k 2^{-n}]\). Then, for \(k 2^{-n} < t \leq (k + 1) 2^{-n}\) we obtain
	\begin{align*}
	X^n_t &= S_{t - k 2^{-n}} \Big( S_{k 2^{-n}} X^n_0 + \int_0^{k 2^{-n}} S_{k 2^{-n} - s} \mu^n (s, X^n, \nu^n_s) ds \\&\hspace{7cm}+ \int_0^{k 2^{-n}} S_{k 2^{-n} -s} \sigma^n  (s, X^n, \nu^n_s) d W_s \Big)  \\&\qquad \qquad +  \int_{k 2^{-n}}^t S_{t - s} \mu^n (s, X^n, \nu^n_s) ds + \int_{k 2^{-n}}^t S_{t-s} \sigma^n (s, X^n, \nu^n_s) d W_s
	\\&= S_t X^n_0 + \int_0^{t} S_{t - s} \mu^n (s, X^n, \nu^n_s) ds + \int_0^t S_{t -s} \sigma^n  (s, X^n, \nu^n_s) d W_s.
	\end{align*}
	Consequently, the proof is complete.
\end{proof}

\noindent
\emph{Step 2: Uniform Moment Bound.} In this step we derive a moment estimate which is useful to establish tightness of the family \(\{X^n \colon n \in \mathbb{N}\}\).
\begin{lemma} \label{lem: moment bound}
	For every \(T > 0\) and every bounded set \(K \subset E\) we have 
	\[
	\sup_{n \in \mathbb{N}}E \Big[\sup_{s \leq T}  \|X^n_s\|^{p'}_E \1_{X^n_0 \in K} \Big] < \infty.
	\]
\end{lemma}
\begin{proof}
	Using \Cref{lem: dynamics,lem: factorization} together with the linear growth assumption~\ref{A3}, we obtain
	\begin{align*}
	E \Big[ \sup_{s \leq t} \|X_s^n\|^{p'}_E \1_{X^n_0 \in K} \Big] 
	&\leq 3^{p' + 1} \Big(\sup_{x \in K} \sup_{s \leq t} \|S_s x\|^{p'}_E + E \Big[ \Big( \int_0^t \|\mu^n (s, X^n, \nu^n_s) \|_E ds \Big)^{p'} \1_{X^n_0 \in K}\Big] \\&\hspace{3.25cm}+ E \Big[ \sup_{s \leq t} \Big\| \int_0^t S_{t - s} \sigma^n (s, X^n, \nu^n_s) \1_{X^n_0 \in K} d W_s \Big\|^{p'} \Big] \Big)
	\\
	&\leq C \Big(1 + E \Big[ \Big( \int_0^t \Big(1 + \sup_{r \leq t} \|X^n_r\|_E\Big) ds \Big)^{p'} \1_{X^n_0 \in K}\Big] \\&\hspace{3.25cm}+ E \Big[ \int_0^t \Big(1 + \sup_{r \leq s} \|X^n_r\|_E \1_{X^n_0 \in K} \Big)^{p'}  ds \Big] \Big)
	\\
	&\leq C \Big(1 + \int_0^t E \Big[ \sup_{r \leq s}\|X_r^n\|^{p'}_E \1_{X^n_0 \in K} \Big] ds \Big)
	\end{align*}
	for all \(t \leq T\).
	As \(t \mapsto  E [\sup_{r \leq t}\|X_r^n\|^{p'}_E]\) is locally bounded on \([0, T]\) by \Cref{lem: dynamics}, Gronwall's lemma (\cite[Lemma 4.4.15]{LS01}) completes the proof.
\end{proof}
~

\noindent
\emph{Step 3: Tightness of \(\{X^n \colon n \in \mathbb{N}\}\).} By the Arzel\`a--Ascoli characterization of tightness (\cite[Theorem~23.4]{Kallenberg}), it suffices to prove that for every \(T > 0\) the family \(\{X^n |_{[0, T]} \colon n \in \mathbb{N}\}\) is tight when seen as Borel probability measures on \(C([0, T], E)\) endowed with the uniform topology. We adapt the compactness method from \cite{doi:10.1080/17442509408833868}.

The factorization formula (see Step 0) and \Cref{lem: dynamics} yield that 
\begin{align}\label{eq: facto}
X^n = S_\cdot X^n_0 + R_1 \mu^n + \tfrac{\sin (\pi \alpha)}{\pi} R_\alpha Y^n, 
\end{align}
where \(\mu^n \equiv (s \mapsto \mu^n(s, X^n, \nu^n_s))\) and 
\[
Y^n_t \triangleq \int_0^t (t - s)^{- \alpha} S_{t-s} \sigma^n(s, X^n, \nu^n_s) d W_s, \quad t \leq T.
\]
Fix \(\varepsilon > 0\). Since a.s. \(X^n_0 = \xi_0 \1_{\|\xi_0\|_E \leq n} \to \xi_0\), we have \(X^n_0 \to \eta\) weakly, which implies that \(\{X^n_0 \colon n \in \mathbb{N}\}\) is tight. Consequently, there exists a compact set \(K \subset E\) such that 
\[
\sup_{n \in \mathbb{N}} P (X^n_0 \not \in K) \leq \frac{\varepsilon}{2}.
\]
Now, we define 
\begin{align*}
K_R \triangleq \Big\{ \omega \in\ &C ([0, T], E)\colon \omega = S x_0 + R_1 \psi + \tfrac{\sin (\pi \alpha)}{\pi}R_\alpha \phi, \\& x_0 \in K, \phi, \psi \in L^{p'} ([0, T], E)  \text{ with } \int_0^T \|\psi (s)\|_E^{p'} ds \leq R, \ \int_0^T \|\phi (s)\|_E^{p'} ds \leq R \Big\}.
\end{align*}
For every \(t \leq T\) the set \(\{S_t x \colon x \in K\}\) is compact by the compactness of the semigroup (and the compactness of \(K\) for \(t = 0\)). By \cite[Lemma I.5.2]{EN00}, the map
\[
[0, T] \times K \ni (t, x) \mapsto S_t x \in E
\]
is uniformly continuous. Thus, the Arzel\`a--Ascoli theorem (\cite[Theorem A.5.2]{Kallenberg}) yields that the set \(\{S x|_{[0,T]} \colon x \in K\}\) is relatively compact in \(C([0, T], E)\), and we conclude from \Cref{lem: R compact} that \(K_R\) is relatively compact in \(C([0, T], E)\). 
Due to \eqref{eq: facto} and Chebyshev's inequality, we have 
\begin{align*}
P ( X^n|_{[0, T]} \in K_R ) &\geq 1 - P(X^n_0 \not \in K) - P \Big( \int_0^T \|\mu^n (s, X^n, \nu^n_s)\|^{p'}_E \1_{X^n_0 \in K} ds > R \Big) \\&\hspace{2cm}-  P \Big( \int_0^T \|Y^n_s\|^{p'}_{E} \1_{X^n_0 \in K} ds > R \Big) 
\\&\geq 1 - \frac{\varepsilon}{2} - \frac{1}{R} \Big( E \Big[ \int_0^T \|\mu^n(s, X^n, \nu^n_s)\|^{p'}_E \1_{X^n_0 \in K} ds \Big] \\&\hspace{2cm}+ E \Big[ \int_0^T \|Y^n_s\|^{p'}_{E}\1_{X^n_0 \in K} ds \Big] \Big).
\end{align*}
By virtue of \eqref{eq: fak bound 2}, \ref{A3} and \Cref{lem: moment bound}, there exists an \(R\) independent of \(n\) such that
\begin{align*}
P ( X^n|_{[0, T]} \in K_R ) \geq 1 - \varepsilon,
\end{align*}
which implies the tightness of the family \(\{X^n|_{[0, T]} \colon n \in \mathbb{N}\}\). Consequently, the family \(\{X^n \colon n \in \mathbb{N}\}\) is tight, too.
\\
~

\noindent
\emph{Step 4: The Cylindrical Martingale Problem.}
By Step 3, we can extract a weakly convergent subsequence from the family \(\{X^n \colon n \in \mathbb{N}\}\). For simplicity, we ignore this subsequence in our notation and assume that \(X^n \to X\) weakly. With little abuse of notation, we write \(P^X_t\) for the law of \(X_t\).

We now study the martingale property of a certain class of test processes.
Take \(g \in C^2_c (\mathbb{R}), f \in C_c(\mathbb{R})\) and \(y^* \in D (A^*)\). The coordinate process on \(C(\mathbb{R}_+, E)\) is denoted by \(\X\). We define 
\[
\Z \triangleq f(\|\X_0\|_E)\Big(g  (\langle \X, y^* \rangle_E ) - g (\langle \X_0, y^*\rangle_E) - \int_0^\cdot \mathcal{L}(s) ds\Big), 
\]
where
\begin{align*}
\mathcal{L} (s) \triangleq \big( \langle \X_s, A^* y^* \rangle_E &+ \langle \mu (s, \X_s, P^X_s), y^* \rangle_E \big) g' (\langle \X_s, y^* \rangle_E)+ \tfrac{1}{2} \|\sigma^* (s, \X_s, P^X_s) y^*\|^2_Hg'' (\langle \X_s, y^*\rangle_E).
\end{align*}
In the following we show that \(\Z\) is a \(P \circ X^{-1}\)-martingale (for the natural filtration of \(\X\)).

For each \(n \in \mathbb{N}\) we define 
\[
Z^n \triangleq f(\|X^n_0\|_E)\Big(g  (\langle X^n, y^* \rangle_E ) - g (\langle X^n_0, y^*\rangle_E) - \int_0^\cdot \mathcal{L}^n (s) ds\Big), 
\]
where
\begin{align*}
\mathcal{L}^n (s) \triangleq \big( \langle X^n_s, A^* y^* \rangle_E &+ \langle \mu^n (s, X^n, \nu^n_s), y^* \rangle_E \big) g' (\langle X^n_s, y^* \rangle_E)\\&+ \tfrac{1}{2} \| \sigma^{n, *} (s,X^n, \nu^n_s) y^*\|^2_Hg'' (\langle X^n_s, y^*\rangle_E).
\end{align*}
Recall the mild dynamics of \(X^n\) as given in \Cref{lem: dynamics}. Thanks to \cite[Theorem 13]{MR2067962}, where we use the second part of \ref{A3}, we can pass to its analytically weak dynamics, i.e. we have 
\begin{align*}
\langle X^n, y^*\rangle_E = \langle X^n_0, y^*\rangle_E & + \int_0^\cdot \big( \langle X^n_s, A^* y^*\rangle_E  + \langle \mu^n(s, X^n, \nu^n_s), y^*\rangle_E \big) ds \\&+ \int_0^\cdot \langle \sigma^{n,*} (s, X^n, \nu^n_s) y^*, d W_s\rangle_H.
\end{align*}
By virtue of these dynamics, It\^o's formula yields that 
\[
Z^n = f (\|X^n_0\|_E) \Big( \int_0^\cdot g' (\langle X^n_s, y^*\rangle_E) d \Big( \int_0^s \langle \sigma^{n, *} (r, X^n, \nu^n_r) y^*, d W_r \rangle_H \Big)\Big).
\]
In particular, \(Z^n\) is a local martingale.
Let \(m > 0\) be such that \(f (x) = 0\) for \(|x| \geq m\). We denote the quadratic variation process by  \([\cdot, \cdot]\) and we deduce from \ref{A3} that
\begin{align*}
E \big[ [Z^n, Z^n]_T \big] &= E \Big[ \int_0^T | f(\|X^n_0\|_E) g' (\langle X^n_s, y^*\rangle_E)|^2 \| \sigma^{n, *} (s, X^n, \nu^n_s) y^*\|^2_Hds \Big]
\\&\leq C \int_0^T E \big[ \|\sigma^n (s, X^n, \nu^n_s)\|^2_{L (H, E)} \1_{\|X^n_0\|_E \leq m} \big] ds 
\\&\leq C \int_0^T E \big[ \|\sigma^n (s, X^n, \nu^n_s)\|^{p'}_{L (H, E)} \1_{\|X^n_0\|_E \leq m} \big]^{2/p'} ds
\\&\leq C \Big(1 + \sup_{s \leq T} E \big[  \|X^n_s\|^{p'}_E \1_{\|X^n_0\|_E \leq m}\big]^{2/ p'}\Big).
\end{align*}
Hence, \(Z^n\) is even a true martingale.
Furthermore, using this estimate in combination with Doob's inequality, we obtain that 
\[
\sup_{n \in \mathbb{N}} \sup_{s \leq T} E \big[ |Z^n_s|^2 \big] \leq C \Big(1 + \sup_{n \in \mathbb{N}} \sup_{s \leq T} E \big[ \|X^n_s\|_E^{p'} \1_{\|X^n_0\| \leq m} \big]^{2/p'}\Big).
\]
As the r.h.s. is finite thanks to \Cref{lem: moment bound}, the family \(\{Z^n_t \colon t \in [0, T], n \in \mathbb{N}\}\) is uniformly integrable.
For every \(t \in \mathbb{R}_+\) the map \(\omega \mapsto \Z_t (\omega)\) is continuous (on \(C(\mathbb{R}_+, E)\) endowed by the local uniform topology)  by assumption \ref{A2}.
By virtue of \cite[Proposition~IX.1.12]{JS}, we can conclude that \(\Z\) is a martingale once we show that 
\[
P ( |Z^n_t - \Z_t (X^n)| \geq \varepsilon) \to 0, \quad \varepsilon, t > 0.
\]
Evidently, we have 
\begin{align*}
|Z^n - \Z (X^n)| &\leq \|f\|_\infty \int_0^\cdot | \mathcal{L}^n (s) - \mathcal{L} (X^n) (s) | ds.
\end{align*}
By Skorokhod's coupling theorem we can assume that \(X\) and \(X^1, X^2, \dots\) are defined on the same probability space and that \(X^n \to X\) almost surely (in the local uniform topology). 
For every \(t > 0\) we have a.s.
\begin{align*}
\|X^n_{\lfloor t 2^n\rfloor 2^{-n}} - X_t\|_E &\leq \|X^n_{\lfloor t 2^n\rfloor 2^{-n}} - X_{\lfloor t 2^n\rfloor 2^{-n}}\|_E + \|X_{\lfloor t 2^n\rfloor 2^{-n}} - X_t\|_E
\\&\leq \sup_{s \leq t} \|X^n_s - X_s\|_E + \|X_{\lfloor t 2^n\rfloor 2^{-n}} - X_t\|_E \to 0.
\end{align*}
Hence, for every \(\phi \in C_b(E)\) and \(t > 0\), the dominated convergence theorem yields that 
\begin{align*}
\int \phi(y) \nu^n_t (dy) = E \big[ \phi(X^n_{\lfloor t 2^n \rfloor 2^{-n}})\big] \to E \big[ \phi (X_t) \big] = \int \phi(y) P^X_t (dy),
\end{align*}
which implies that \(\nu^n_t \to P^X_t\) weakly.
Take \(\omega, \omega^1, \omega^2, \dots \in C(\mathbb{R}_+, E)\) such that \(\omega^n \to \omega\) in the local uniform topology. By the Arzel\`a--Ascoli theorem (\cite[Theorem~A.5.2]{Kallenberg}), there exists a compact set \(K \subset E\) such that \(\omega^n (s) \in K\) for all \(s \in [0, T]\) and \(n \in \mathbb{N}\), and 
\begin{align}\label{eq: part AA theo}
\sup_{n \in \mathbb{N}} \sup\big\{ \|\omega^n (s) - \omega^n(r)\|_E \colon s, r \in [0, T], |s - r| \leq h\big\} \to 0 \text{ as } h \searrow 0.
\end{align}
Fix \(t \leq T\) and \(\varepsilon > 0\) and denote by \(d_c\) a metric which induces the topology on \(M_c(E)\), i.e. the topology of convergence in distribution.  By \ref{A2}, the function \(\langle \mu (t, \cdot, \cdot), y^*\rangle_E\) is uniformly continuous on the compact set \(G \triangleq K \times \{P^X_t, \nu^1_t, \nu^2_t, \dots\}\). 
Thus, there exists a \(\delta > 0\) such that 
\begin{align*}
(x, \nu), (y, \eta) \in G,\ \  \|x - y\|_E &+ d_c(\nu, \eta) \leq \delta \\& \Longrightarrow \ |\langle \mu (t, x, \nu) - \mu (t, y, \eta), y^*\rangle_E| \leq \varepsilon.
\end{align*}
As \(\nu^n_t \to P^X_t\) weakly, there exists an \(N \in \mathbb{N}\) such that \(d_c(\nu^n_t, P^X_t) \leq \tfrac{\delta}{2}\) for all \(n \geq N\).
Furthermore, thanks to \eqref{eq: part AA theo}, there exists an \(M \in \mathbb{N}\) such that for all \(n \geq M\)
\[
\|\omega^n (\lfloor t 2^n\rfloor 2^{-n}) - \omega^n (t)\|_E \leq \tfrac{\delta}{2}.
\]
Thus, for all \(n \geq N \vee M\) we have 
\[
|\langle \mu (t, \omega^n (\lfloor t 2^n\rfloor 2^{-n}), \nu^n_t) - \mu (t, \omega^n(t), P^X_t), y^*\rangle_E| \leq \varepsilon.
\]
We conclude that a.s. for all \(s \in [0, T]\)
\begin{align*}
 |\langle \mu^n(s&,X^n, \nu^n_s) - \mu (s, X^n_s, P^X_s), y^*\rangle_E| 
 \to 0.
\end{align*}
Now, using the same argument for the coefficient \(\sigma\) and the dominated convergence theorem (which we can use thanks to \ref{A3}), we obtain that a.s. 
\[
|Z^n_t - \Z_t (X^n)| \to 0.
\]
We conclude that \(\Z\) is a \(P \circ X^{-1}\)-martingale. 
\\
~~

\noindent
\emph{Step 5: Conclusion.} We are in the position to complete the proof. 
Take \(y^* \in D(A^*)\) and define
\begin{align*}
U_N &\triangleq \inf (t \in \mathbb{R}_+ \colon |\langle \X_t, y^*\rangle_E| \geq N), \\ S_N &\triangleq \begin{cases} 0,& \|\X_0\|_E > N,\\
+ \infty, & \|\X_0\|_E \leq N, \end{cases} \\T_N &\triangleq U_N \wedge S_N,
\end{align*}
for \(N \in \mathbb{N}\). Clearly, by the continuous paths of \(\X\), \(T_N\) is a stopping time for the filtration generated by \(\X\).
Using Step~4 with  \(g \in C^2_c(\mathbb{R})\) such that \(g (x) = x\) for all \(|x| \leq N\) and \(f \in C_c(\mathbb{R})\) such that \(f (x) = 1\) for all \(|x| \leq N\) yields that the process 
\[
\langle \X_{\cdot \wedge T_N}, y^*\rangle_E - \langle \X_0, y^*\rangle_E - \int_0^{\cdot \wedge T_N} \big( \langle \X_s, A^* y^*\rangle_E + \langle \mu (s, \X_s, P^X_s), y^* \rangle_E \big) ds
\]
is a \(P \circ X^{-1}\)-martingale. Consequently, since \(T_N \nearrow \infty\) as \(N \to \infty\), the process
\[
\langle \X, y^*\rangle_E - \langle \X_0, y^*\rangle_E- \int_0^\cdot \big( \langle \X_s, A^* y^*\rangle_E + \langle \mu (s, \X_s, P^X_s), y^* \rangle_E \big) ds
\]
is a local \(P \circ X^{-1}\)-martingale. By virtue of the proof of \cite[Proposition 5.4.6]{KaraShre}, using the same argument with \(g \in C^2_c(\mathbb{R})\) such that \(g (x) = x^2\) for \(|x| \leq N\) yields that its quadratic variation process is given by
\(
\int_0^\cdot \| \sigma^* (s, \X_s, P^X_s) y^*\|^2_H ds.
\)
Recall that \(y^* \in D(A^*)\) was arbitrary.
As \(D(A^*)\) separates points, the representation theorem \cite[Theorem 3.1]{doi:10.1142/S0219025707002816} yields the existence of a standard cylindrical Brownian motion \(B\) defined on a standard extension of \((C(\mathbb{R}_+, E), \mathcal{B}(C(\mathbb{R}_+,E)), P \circ X^{-1})\) with the canonical filtration generated by \(\X\), such that 
\begin{align*}
\langle \X, y^*\rangle_E - \langle \X_0, y^*\rangle_E - \int_0^\cdot \big( \langle \X_s, A^* y^*\rangle_E + \langle \mu &(s, \X_s, P^X_s), y^* \rangle_E \big) ds \\&= \int_0^\cdot \langle \sigma^* (s, \X_s, P^X_s) y^*, d B_s \rangle_H
\end{align*}
for all \(y^* \in D(A^*)\). 

Finally, noting that the initial value \(\X_0\) is distributed according to \(\eta\) under the probability measure \(P \circ X^{-1}\), we conclude that \(\X\) is an analytically weak solution process to the MKV SPDE with coefficients \((A, \mu, \sigma, \eta)\). By \cite[Theorem 13]{MR2067962} it is also a mild solution process and the existence of a martingale solution is proved.
\qed

\section{Proof of \Cref{theo: main wasserstein cont}} \label{sec: pf outline wasserstein}
\Cref{theo: main wasserstein cont} can be proved similar to \Cref{theo: main} and in the following outline the few necessary changes where we use the notation from \Cref{sec: pf theo 1}.
In case \(\eta \in M^{p'}_w(E)\), \Cref{lem: moment bound} holds for \(K = E\) and therefore, \(\nu^n_t \to P^X_t\) in \(M^{p^\circ}_w(E)\) by \cite[Theorem 5.5]{Carmona18} as \(p^\circ < p'\). Furthermore, in the definition of \(\mathsf{Z}\) and \(Z^n\) (see Step 4 in the proof of \Cref{theo: main}) we can take \(f \equiv 1\). In Step~5 it is not necessary to introduce \(S_N\) which means one can use \(T_N \equiv U_N\).  
Finally, we explain why \((t \mapsto P^X_t) \in C(\mathbb{R}_+, M^{p'}_w(E))\). Let \(X^1, X^2, \dots\) be the approximation sequence and let \(X\) be an accumulation point, i.e. for simplicity assume that \(X^n \to X\) weakly. Then, as \Cref{lem: moment bound} holds for \(K = E\), we obtain
\[
E \Big[ \sup_{s \leq T} \|X_s\|^{p'}_E \Big] \leq \liminf_{n \to \infty} E \Big[ \sup_{s \leq T} \|X^n_s\|^{p'}_E\Big] \leq \sup_{n \in \mathbb{N}} E \Big[ \sup_{s \leq T} \|X^n_s\|^{p'}_E\Big] < \infty.
\]
Since a.s. \(\|X_s\|^{p'}_E \to \|X_t\|^{p'}_E\) for \(s \to t\) by the continuous paths of \(X\), the dominated convergence theorem yields that \((t \mapsto P^X_t) \in C(\mathbb{R}_+, M^{p'}_w(E))\).
No further changes are needed. \qed

\section{Proof of \Cref{theo: uni}} \label{sec: pf theo uni}
The basic strategy of proof is borrowed from  \cite[Theorem 3.1]{Funkai84}. Let \(\gamma \in C(\mathbb{R}_+, M^p_w(E))\). 
We now define solutions to a certain class of classical SPDEs.
\begin{definition} \label{def: SPDE classical}
	We call a triplet \((\B, W, X)\) a \emph{martingale solution} to the SPDE with coefficients \((A, \mu, \sigma, \gamma, \eta)\) if \(\B\) is a filtered probability space with right-continuous and complete filtration which supports a standard cylindrical Brownian motion \(W\) and a continuous \(E\)-valued adapted process \(X\) such that the following hold:
	\begin{enumerate}
		\item[\textup{(i)}] \(X_0 \sim \eta\).
		\item[\textup{(ii)}]	 Almost surely for all \(t \in \mathbb{R}_+\)
		\begin{align*}
		\int_0^t \|S_{t - s} \mu(s, X_s, \gamma(s))\|_E ds + \int_0^t \|S_{t - s} \sigma(s, X_s, \gamma(s))\|_{L_2 (H, E)}^2 ds < \infty.
		\end{align*}
		\item[\textup{(iii)}]	Almost surely for all \(t \in \mathbb{R}_+\)
		\begin{equation*}
		\begin{split}
		X_t = S_t X_0 + \int_0^t S_{t - s} \mu(s, X_s, \gamma(s)) ds + \int_0^t S_{t - s} \sigma(s, X_s, \gamma (s)) d W_s.
		\end{split}
		\end{equation*}
	\end{enumerate}
\end{definition}

The proof of the following lemma is given at the end of this section. 
\begin{lemma} \label{lem: uni main tech}
	For \(i = 1, 2\), let \((\B^i, W^i, X^i)\) be a martingale solution to the SPDE with coefficients \((A, \mu, \sigma, \gamma^i, \eta)\) and let \(u^i (t)\) be the law of \(X^i_t\) for \(t \in \mathbb{R}_+\). Furthermore, assume that \(u^i \in C(\mathbb{R}_+, M_w^p(E))\) for \(i = 1, 2\). Then, for every \(T > 0\) and \(m > 0\) such that 
	\[
	\max_{i = 1, 2} \sup_{s \leq T} \|\gamma^i(s)\|_p \leq m, 
	\]
	 there exists a constant \(C = C(p, T, S, m) > 0\) such that 
	\[
	|\mathsf{w}_p (u^1 (s), u^2(s))|^p \leq C \int_0^s | \kappa (\mathsf{w}_p(\gamma^1(r), \gamma^2(r)))|^p dr, \qquad 0 \leq s \leq T,
	\]
	where \(\kappa = \kappa_{T, m}\) is as in \ref{U1} or \ref{U2}, respectively.
\end{lemma}

In the following we prove \Cref{theo: uni}. For contradiction, let \((\B^i, W^i, X^i), i = 1, 2,\) be two \(p\)-martingale solutions to the MKV SPDE \((A, \mu, \sigma, \eta)\) such that \(P^1 \circ (X^1)^{-1} \not = P^2 \circ (X^2)^{-1}\). We define \(u^i (t) \triangleq P^i \circ (X^i_t)^{-1}\) for \(t \in \mathbb{R}_+\) and \(i = 1, 2\). By definition of a \(p\)-martingale solution, we have \(u^i \in C(\mathbb{R}_+, M^p_w(E))\) for \(i = 1,2\).

\begin{lemma} \label{lem: s}
	\(\s \triangleq \inf (t \in \mathbb{R}_+ \colon u^1(t) \not = u^2(t)) < \infty\).
\end{lemma}
\begin{proof}
	For contradiction, assume that \(\s = \infty\), i.e. \(u^1 = u^2 \equiv u\). Then, \((B^i, W^i, X^i)\) are both martingale solutions to the SPDE \((A, \mu, \sigma, u, \eta)\). Thanks to \Cref{theo: DZ} (under \(p > 1/\alpha\) and \ref{U1}) or \cite[Theorem 7.2]{DePrato} (under \(p \geq 2\) and \ref{U2}), this SPDE has a pathwise unique solution.
		 By the Yamada--Watanabe theorem \cite[Theorem 2]{MR2067962}, the SPDE also satisfies uniqueness in law. This contradicts \(P^1 \circ (X^1)^{-1} \not = P^2 \circ (X^2)^{-1}\) and the claim follows. 
\end{proof}
\begin{remark} 
The conclusion of \Cref{lem: s} might be compared to \cite[Theorem 4.4.2]{EK} which shows that two solutions of certain (time-homogeneous) martingale problems have the same law already if they have the same one-dimensional distributions. 
\end{remark}
\Cref{lem: uni main tech} yields that 
\begin{align*}
|\mathsf{w}_p (u^1 (s), u^2(s))|^p \leq C \int_0^s | \kappa (\mathsf{w}_p(u^1(r), u^2(r)))|^p dr, \qquad 0 \leq s \leq \s + 1,
\end{align*}
where \(\kappa\) depends on \(\s, u^1\) and \(u^2\).
Hence, recalling the properties of \(\kappa\), Bihari's lemma (\cite[Lemma~5.2.8]{roeckner15}) implies that \(u^1 = u^2\) on \([0, \s + 1]\). However, as this contradicts the definition of~\(\s\), we can conclude that the MKV SPDE satisfies \(p\)-uniqueness in law. 

Finally, let \((\B, W, X)\) and \((\B, W, Y)\) be two \(p\)-martingale solutions to the MKV SPDE \((A, \mu, \sigma, \eta)\). By the previous part of the proof, we know that \(X\) and \(Y\) have the same law. We write \(\gamma (t) \triangleq P^X_t = P^Y_t\) for \(t \in \mathbb{R}_+\). Now, \((\B, W, X)\) and \((\B, W, Y)\) both are martingale solutions to the SPDE \((A, \mu, \sigma, \gamma, \eta)\). Consequently, as this SPDE satisfies pathwise uniqueness by \Cref{theo: DZ} or \cite[Theorem 7.2]{DePrato}, we have a.s. \(X = Y\). The proof of \Cref{theo: uni} is complete. \qed
\\

\noindent
It remains to prove \Cref{lem: uni main tech}. 
\begin{proof} [Proof of \Cref{lem: uni main tech}]
	Thanks to \Cref{theo: DZ} or \cite[Theorem 7.2]{DePrato}, the SPDE can be realized on any driving system, and, thanks to the Yamada--Watanabe theorem \cite[Theorem 2]{MR2067962}, it also satisfies uniqueness in law. Consequently, we can w.l.o.g. assume that \((\B^1, W^1) = (\B^2, W^2) \equiv (\B, W)\).
	Take \(T > 0\) and let \(m > 0\) be such that \(\sup_{s \leq T} \|\gamma^i (s)\|_p \leq m\). For \(t \leq T\) we have 
	\begin{equation*} 
	\begin{split}
	E \big[  \|X^1_s - X^2_s\|^p_E \big] & \leq 2^{p - 1} E \Big[ \Big\| \int_0^s S_{s- r}( \mu(r, X^1_r, \gamma^1 (r)) - \mu (r, X^2_r, \gamma^2 (r)) ) dr \Big\|^p_E \Big]
	\\&\hspace{0.75cm} + 2^{p - 1}E \Big[ \Big\| \int_0^s S_{s - r} ( \sigma(r, X^1_r, \gamma^1 (r)) - \sigma (r, X^2_r, \gamma^2 (r)) ) d W_r \Big\|^p_E \Big].
	\end{split}
	\end{equation*}
	We now estimate each of the latter terms separately, starting with the second term. 
	In case \(p > 1/\alpha\) and \ref{U1} hold, using \Cref{lem: factorization} yields that
	we get
		\begin{equation*}
		\begin{split}
	E \Big[   \Big\| \int_0^t &S_{t - r} (\sigma (r, X^1_r, \gamma^1 (r)) - \sigma (r, X^2_r, \gamma^2 (r))) d W_r \Big\|^p_E \Big] 
	\\&\leq C E \Big[ \int_0^t \big( \|X^1_r - X^2_r\|^p_E + |\kappa (\mathsf{w}_p(\gamma^2(r), \gamma^2(r)))|^p \big) dr  \Big].
	\end{split}
	\end{equation*}
	Using Burkholder's inequality instead of \Cref{lem: factorization}, the same inequality holds in case \(p \geq 2\) and \ref{U2} hold.
Using H\"older's inequlity and \ref{U1}, we also get that
	\begin{align*}
	E \Big[  \Big\| \int_0^t &S_{t- r}( \mu(r, X^1_r, \gamma^1 (r)) - \mu (r, X^2_r, \gamma^2 (r)) ) dr \Big\|^p_E \Big] 
	\\&\leq E \Big[  \Big( \int_0^t \|S_{t- r}( \mu(r, X^1_r, \gamma^1 (r)) - \mu (r, X^2_r, \gamma^2 (r)) )\|_E dr \Big)^p \Big] 
	\\&\leq E \Big[  \Big( \int_0^t \g (t - r) \big(\|X^1_r - X^2_r\|_E + \kappa (\mathsf{w}(\gamma^1(r), \gamma^2(r)))\big) dr \Big)^p \Big] 
	\\&\leq \Big( \int_0^T \big[ \g (s) \big]^{p/(p - 1)} ds \Big)^{p - 1} E \Big[  \int_0^t \big(\|X^1_r - X^2_r\|_E + \kappa (\mathsf{w}(\gamma^1(r), \gamma^2(r)))\big)^p dr  \Big] 
	\\&\leq C E \Big[ \int_0^t \big(  \|X^1_r - X^2_r\|_E^p + |\kappa (\mathsf{w}_p(\gamma^1(r), \gamma^2(r)))|^p \big)dr \Big],
	\end{align*}
	with \(\g = \g_{T, m}\) as in \ref{U1}.
	A similar computation gives the inequality 
		\begin{align*}
	E \Big[  \Big\| \int_0^t &S_{t- r}( \mu(r, X^1_r, \gamma^1 (r)) - \mu (r, X^2_r, \gamma^2 (r)) ) dr \Big\|^p_E \Big] 
	\\&\leq C E \Big[ \int_0^t \big(  \|X^1_r - X^2_r\|_E^p + |\kappa (\mathsf{w}_p(\gamma^1(r), \gamma^2(r)))|^p \big)dr \Big],
	\end{align*}
	under \ref{U2}.
	Putting these estimates together, we obtain for all \(t \leq T\)
	\[
		E \big[ \|X^1_t - X^2_t\|^p_E \big] \leq C \Big( \int_0^t E \big[ \|X^1_s - X^2_s\|^p_E \big] ds + \int_0^t |\kappa (\mathsf{w}_p(\gamma^1(r), \gamma^2(r)))|^p dr \Big).
	\] 
	As \(t \mapsto E [ \|X^1_t - X^2_t\|^p_E]\) is locally bounded, we deduce from Gronwall's lemma that
	\begin{align*}
	E \big[ \|X^1_t - X^2_t\|^p_E \big] &\leq C e^{C T} \int_0^t |\kappa (\mathsf{w}_p(\gamma^1(r), \gamma^2(r)))|^p dr
	\end{align*}
	for all \(t \leq T\). 
	Finally, the claim follows from the observation that
	\[
	|\mathsf{w}_p (u^1(t), u^2(t))|^p \leq E \big[ \|X^1_t - X^2_t\|^p_E \big], \quad t \leq T.
	\]
	The proof is complete.
\end{proof}

\section{Proof of \Cref{theo: POC new1}} \label{sec: pf POC new 1}
Throughout the proof we fix a finite time horizon \(T > 0\). Except in the final step, all processes in the following are meant to be defined on the finite time interval \([0, T]\).
\\

\noindent
\emph{Step 1: Tightness in \(M_c(M_c(C([0, T], E)))\).}  We adapt the argument from Step 3 in the proof of \Cref{theo: main}.
In the first part of this step we establish a uniform moment bound.
Recall the notation that for \(x = (x_1, \dots, x_N) \in E^{\otimes N}\) we have
\[
L^N (x) = \frac{1}{N} \sum_{i = 1}^N \delta_{x_i}.
\]
By virtue of \ref{A5}, for all \(0 < s, t \leq T, x = (x_1, \dots, x_N) \in E^{\otimes N}\) and \(i = 1, \dots, N\) we have 
\begin{equation} \label{eq: linear growth POC} \begin{split}
\|\mu (t, x_i, L^N (x))\|_E^{p'}  &\leq C \Big(1 + \|x_i\|^{p'}_E + \|L^N (x)\|^{p'}_{p'}\Big) 
\\&= C \Big(1 + \|x_i\|^{p'}_E + \frac{1}{N} \sum_{j= 1}^N \|x_j\|_E^{p'} \Big).
\end{split}
\end{equation}
Let 
\[
T_m \triangleq \inf \Big(t \leq T \colon \frac{1}{N}\sum_{i = 1}^N \|X^{N, i}_t\|_E^{p'} \geq m\Big), \quad m > 0.
\]
By virtue of \ref{A5} and \Cref{lem: factorization}, we obtain that 
\begin{equation} \label{eq: bound convolution empirical measure}
\begin{split}
E \Big[ \sup_{s \leq t \wedge T_m} \Big\| \int_0^s &S_{s- r} \sigma (r, X^{N, i}_r, \mathcal{X}^N_r) dW^i_r \Big\|^{p'}_E \Big] 
\\&\leq E \Big[ \sup_{s \leq t} \Big\| \int_0^s S_{s- r} \sigma (r, X^{N, i}_{r \wedge T_m}, L^N(X^{N, 1}_{r \wedge T_m}, \dots, X^{N, N}_{r \wedge T_m})) dW^i_r \Big\|^{p'}_E \Big]
\\&\leq C E \Big[ \int_0^{t} \Big(1 + \|X^{N, i}_{s \wedge T_m}\|^{p'}_E + \frac{1}{N} \sum_{j = 1}^N \|X^{N, j}_{s \wedge T_m}\|^{p'}_E \Big) ds \Big].
\end{split}
\end{equation}
Now, thanks to \eqref{eq: linear growth POC}, \eqref{eq: bound convolution empirical measure} and the uniform moment bound on the initial values from \ref{I}, for all~\(t \leq T\) we obtain
\begin{align*}
E \Big[ \sup_{s \leq t \wedge T_m} \frac{1}{N} \sum_{i = 1}^N \|X^{N, i}_s\|^{p'}_E \Big]
&\leq \frac{1}{N} \sum_{i = 1}^N E \Big[ \sup_{s \leq t \wedge T_m} \|X^{N, i}_s\|^{p'}_E \Big] 
\\&\leq \frac{1}{N} \sum_{i = 1}^N C \Big( 1 + E \Big[ \int_0^{t} \Big(\|X^{N, i}_{s \wedge T_m}\|^{p'}_E + \frac{1}{N} \sum_{j = 1}^N \|X^{N, j}_{s \wedge T_m}\|^{p'}_E \Big) ds \Big] \Big)
\\
&= C \Big(1 + 2 E \Big[ \int_0^{t} \frac{1}{N} \sum_{i = 1}^N \|X^{N, i}_{s \wedge T_m}\|^{p'}_E ds \Big] \Big)
\\&\leq C \Big(1 + \int_0^t E \Big[ \sup_{r \leq s \wedge T_m} \frac{1}{N} \sum_{i = 1}^N \|X^{N, i}_r\|^{p'}_E \Big] ds \Big).
\end{align*}
As, by definition of \(T_m\), for all \(t \leq T\)
\begin{align*}
E \Big[\sup_{s \leq t \wedge T_m} \frac{1}{N} \sum_{i = 1}^N \|X^{N, i}_s\|^{p'}_E \Big] &\leq \frac{1}{N} \sum_{i = 1}^N E \big[ \|X^{N, i}_0\|^{p'}_E \big] + m
\\&\leq \sup_{n \in \mathbb{N}} \int \|\X^N_1 (x)\|^{p'}_E \eta^n( dx ) + m, 
\end{align*}
we deduce from Gronwall's lemma that 
\[
E \Big[\sup_{s \leq T \wedge T_m} \frac{1}{N} \sum_{i = 1}^N \|X^{N, i}_s\|^{p'}_E \Big]  \leq C.
\]
Hence, letting \(m \to \infty\), Fatou's lemma yields that 
\[
E \Big[ \sup_{s \leq T} \|X^{N, 1}_s\|^{p'}_E\Big] \leq E \Big[\sup_{s \leq T} \sum_{i = 1}^N \|X^{N, i}_s\|^{p'}_E \Big]  \leq CN.
\]
Next, using that \(X^{N, i}\) and \(X^{N, j}\) have the same law for all \(i, j \leq N\) by assumption \ref{E}, and the uniform moment bound on the initial values from \ref{I}, arguing as above, we get for all \(t \leq T\) that
\begin{align*}
E \Big[ \sup_{s \leq t} \|X^{N, 1}_s\|^{p'}_E \Big] 
&\leq C \Big(1 + \int_0^t E \Big[ \sup_{r \leq s}\|X^{N, 1}_r\|^{p'}_E \Big] ds \Big).
\end{align*}
Thus, as \(E \big[ \sup_{s \leq T} \|X^{N, 1}_s\|^{p'}_E\big] < \infty\), we can apply Gronwall's lemma again and obtain that
\begin{align}\label{eq: moment bound prop chaos pf}
\sup_{N \in \mathbb{N}}  E \Big[ \sup_{s \leq T}\|X^{N, 1}_s\|^{p'}_E \Big] \leq C.
\end{align}

We are in the position to deduce tightness. Fix \(\varepsilon > 0\).
As the empirical distributions \(\mathcal{X}^N_0\) converge weakly as \(N \to \infty\) by assumption \ref{I}, \cite[Proposition I.2.2]{10.1007/BFb0085169} yields that the family \(\{X^{N, 1}_0 \colon N \in \mathbb{N}\}\) is tight. Thus, there exists a compact set \(K \subset E\) such that 
\begin{align}\label{eq: eps tight}
\sup_{N\in\mathbb{N}} P \big(X^{N, 1}_0 \not \in K\big) \leq \frac{\varepsilon}{2}.
\end{align}
Recalling the notation from Step 0 in the proof of \Cref{theo: main}, the factorization formula (see again Step 0 in the proof of \Cref{theo: main}) yields that
\begin{align*}
X^{N, 1} = S_\cdot X_0 + R_1 \mu^{N} + \tfrac{\sin (\pi \alpha)}{\pi} R_\alpha Y^{N}, 
\end{align*}
where \(\mu^{N} \equiv (s \mapsto \mu(s, X^{N, 1}_s, \mathcal{X}^N_s))\) and 
\[
Y^{N}_t \triangleq \int_0^t (t - s)^{- \alpha} S_{t-s} \sigma(s, X^{N, 1}, \mathcal{X}^{N}_s) d W^1_s, \quad t \leq T.
\]
Now, we define 
\begin{align*}
K_R \triangleq \Big\{ \omega \in\ &C ([0, T], E)\colon \omega = Sx_0 + R_1 \psi + \tfrac{\sin (\pi \alpha)}{\pi}R_\alpha \phi,\ \\& x_0 \in K, \phi, \psi \in L^{p'} ([0, T], E)  \text{ with } \int_0^T \|\psi (s)\|_E^{p'} ds \leq R, \ \int_0^T \|\phi (s)\|_E^{p'} ds \leq R \Big\}.
\end{align*}
Thanks to \Cref{lem: R compact} and the compactness of the semigroup \(S\), the set \(K_R\) is relatively compact in \(C([0, T], E)\). 
Using \eqref{eq: fak bound 2}, \ref{A5} and the assumption that \(X^{N, i}\) and \(X^{N, j}\) have the same law for all \(i, j \leq N\), we estimate
\begin{equation} \label{eq: eps tight 2} \begin{split}
E \Big[ \int_0^T \|Y^{N}_s\|^{p'}_E ds \Big] 
&\leq C \Big(1 + \sup_{N \in \mathbb{N}}\sup_{s \leq T} E \big[ \|X^{N, 1}_s\|^{p'}_E\big] \Big).
\end{split}
\end{equation}
Similarly, thanks to \eqref{eq: linear growth POC}, we obtain that 
\begin{align} \label{eq: eps tight 3}
E \Big[ \int_0^T \|\mu (s, X^{N, 1}_s, \mathcal{X}^N_s) \|^{p'}_E ds \Big] \leq C \Big(1 + \sup_{N \in \mathbb{N}}\sup_{s \leq T} E \big[ \|X^{N, 1}_s\|^{p'}_E\big] \Big).
\end{align}
In summary, using Chebyshev's inequality and \eqref{eq: eps tight}, \eqref{eq: eps tight 2} and \eqref{eq: eps tight 3}, for every \(\varepsilon> 0\) we can take \(R = R(\varepsilon) > 0\) large enough such that 
\[
P \big(X^{N, 1} \in K_R\big) \geq 1 - \varepsilon.
\]
Consequently, the family \(\{X^{N, 1} \colon N \in \mathbb{N}\}\) is tight.\\ 

\noindent
\emph{Step 2: Tightness in \(M_c(M^{p^\circ}_w(C([0, T], E)))\).}
Let \(d_T\) be the uniform metric on \(C([0, T], E)\),~i.e. 
\[
d_T (\omega, \alpha) \triangleq \sup_{s \leq T} \|\omega(s) - \alpha(s)\|_E, \quad \omega, \alpha \in C([0, T], E).
\]
In the following we consider \(\mathcal{X}^1, \mathcal{X}^2, \dots\) as random variables with values in \(M_w^{p^\circ}(C([0, T], E))\). 
Next, we show that \(\{\mathcal{X}^N \colon N \in \mathbb{N}\}\) is tight. 

Fix \(\varepsilon > 0\) and define
\[
a_n \triangleq n^{1/(p' - p^\circ)} 2^{n/(p' - p^\circ)},\qquad b_n \triangleq \varepsilon n\Big/ \Big[\sup_{m \in \mathbb{N}} E \Big[ \sup_{s \leq T} \|X^{m, 1}_s\|^{p'}_E \Big] \vee 1 \Big],
\]
and
\begin{align} \label{eq: K eps}
K_\varepsilon \triangleq \bigcap_{n \in \mathbb{N}}\Big\{ \nu \in M^{p^\circ}_w (C([0, T], E)) \colon \int d_T (\omega, 0)^{p^\circ} \1_{d_T (\omega, 0) \geq a_n}\nu(d \omega) < \frac{1}{b_n} \Big\}.
\end{align}
For every \(N \in \mathbb{N}\) we obtain
\begin{align*}
P \big(\mathcal{X}^N \not \in K_\varepsilon\big) &\leq \sum_{n = 1}^\infty P \Big( \frac{1}{N} \sum_{i = 1}^N d_T (X^{N, i}, 0)^{p^\circ} \1_{d_T (X^{N, i}, 0) \geq a_n} \geq \frac{1}{b_n}\Big)
\\&\leq \sum_{n = 1}^\infty \frac{b_n}{N} \sum_{i = 1}^N E \big[d_T (X^{N, i}, 0)^{p^\circ} \1_{d_T (X^{N, i}, 0) \geq a_n}\big]
\\&\leq \sum_{n = 1}^\infty \frac{b_n}{a_n^{p' - p^\circ}} E \Big[ \sup_{s \leq T} \|X^{N, 1}_s\|^{p'}_E \Big]
\\&\leq \varepsilon.
\end{align*}

By \cite[Proposition I.2.2]{10.1007/BFb0085169}, Step 1 implies that the family \(\{\mathcal{X}^N \colon N \in \mathbb{N}\}\), seen as random variables in \(M_c (C([0, T], E))\), is tight. 
Consequently, as \(\{\mathcal{X}^N \colon N \in \mathbb{N}\} \subset M_w^{p^\circ} (C([0, T], E))\), there exists a set \(G_\varepsilon \subset M^{p^\circ}_w(C([0, T], E))\) which is relatively compact in \(M_c(C([0, T], E))\) such that 
\[
\sup_{N \in \mathbb{N}}P \big(\mathcal{X}^N \not \in G_\varepsilon\big) \leq \varepsilon.
\]
Let \(K_\varepsilon\) be as in \eqref{eq: K eps}. Then, we have for all \(N \in \mathbb{N}\)
\[
P\big(\mathcal{X}^N \not \in G_\varepsilon \cap K_\varepsilon\big) \leq P\big(\mathcal{X}^N \not \in G_\varepsilon\big) + P\big(\mathcal{X}^N \not \in K_\varepsilon\big) \leq 2 \varepsilon.
\]
Using that the set \(G_\varepsilon \cap K_\varepsilon\) is relatively compact in \(M_w^{p^\circ}(C([0, T], E))\) by \cite[Corollary 5.6]{Carmona18}, we can conclude that the family \(\{\mathcal{X}^N \colon N \in \mathbb{N}\}\) is tight when seen as random variables with values in \(M_w^{p^\circ}(C([0, T], E))\).
From now on we assume that \(\mathcal{X}^N \to \mathcal{X}\) in \(M_c(M_w^{p^\circ}(C([0, T], E)))\), i.e  that the laws of \(\mathcal{X}^N\), which are considered as elements of \(M_c(M^{p^\circ}_w(C([0, T], E)))\), converge weakly to the law of \(\mathcal{X} \in M^{p^\circ}_w(C([0, T], E))\).
We note that 
\begin{equation} \label{eq: moment ineq final} \begin{split}
E \Big[ \int d_T (\omega, 0)^{p'} \mathcal{X}(d \omega) \Big] &\leq \liminf_{N \to \infty} E \Big[ \int d_T (\omega, 0)^{p'} \mathcal{X}^N(d \omega) \Big] 
\\&= \liminf_{N \to \infty} \frac{1}{N} \sum_{i = 1}^N E \Big[ \sup_{s \leq T} \|X^{N, i}_s\|^{p'}_E \Big] 
\\&= \liminf_{N \to \infty} E\Big[ \sup_{s \leq T} \|X^{N, 1}_s\|^{p'}_E\Big]
\\&\leq \sup_{N \in \mathbb{N}} E \Big[ \sup_{s \leq T} \|X^{N, 1}_s\|^{p'}_E\Big] < \infty.
\end{split}
\end{equation}
Thus, a.s. \(\mathcal{X} \in M^{p'}_w (C([0, T], E)) \subset M_w^{p^\circ} (C([0, T], E))\).
\\

\noindent
\emph{Step 3: Convergence of Test Processes.}
Take \(0 \leq s < t \leq T, t_1, \dots, t_m \in [0, s], h_1, \dots, h_m \in C_b(E), g \in C^2_c(\mathbb{R})\) and \(y^* \in D(A^*)\). 
For \((\omega, \nu) \in C([0, T], E) \times M^{p^\circ}_w(C([0, T], E))\) we define
\begin{align*}
M_r(\omega, \nu) &\triangleq g(\langle \omega (r), y^*\rangle_E) - g(\langle \omega(0), y^*\rangle_E) - \int_0^r \mathcal{L}_u (\omega , \nu) du , \quad r \leq T,
\end{align*}
where
\begin{align*}
\mathcal{L}_u (\omega , \nu) \triangleq \big( \langle \omega(u), A^* y^* \rangle_E &+ \langle \mu (u, \omega(u), \nu \circ \X_u^{-1}), y^* \rangle_E \big) g' (\langle \omega(u), y^* \rangle_E)\\&+ \tfrac{1}{2} \|\sigma^* (u, \omega(u), \nu \circ \X^{-1}_u) y^*\|^2_H g'' (\langle \omega(u), y^*\rangle_E),
\end{align*}
and
\[
V (\omega, \nu) \triangleq \big(M_t(\omega, \nu) - M_s(\omega, \nu)\big) \prod_{i = 1}^m h_i (\omega(t_i)).
\]
For all  \(\nu \in M_w^{p^\circ} (C([0, T], E))\) and \(k > 0\) we define
\[
Z_k (\nu) \triangleq \int \big[(-k) \vee V(\omega, \nu) \wedge k \big] \nu(d\omega), \quad Z (\nu) \triangleq \liminf_{k \to \infty} Z_k (\nu).
\]
If \(\nu_n \to \nu\) in \(M_w^{p^\circ}(C([0, T], E))\), then \((r \mapsto \nu_n \circ \X_r^{-1}) \to (r \mapsto \nu \circ \X_r^{-1})\) in \(C([0, T], M^{p^\circ}_w(E))\), which follows from the inequality
\[
\sup_{r \leq T} \mathsf{w}_{p^\circ} (\nu_n \circ \X_r^{-1}, \nu \circ \X_r^{-1}) \leq \mathsf{w}^{p^\circ}_T (\nu_n, \nu),
\]
where \(\mathsf{w}^{p^\circ}_T\) is the \(p^\circ\)-Wasserstein metric on \(M_w^{p^\circ}(C([0, T], E))\).
By the continuity assumption \ref{C} and the dominated convergence theorem (which is applicable due to the local boundedeness part in \ref{C}), \(V\) is continuous. Thus, \cite[Theorem 8.10.61]{Bogachev07} yields that \(Z_k\) is continuous. In particular, \(Z\) is Borel measurable.

For all \((\omega, \nu) \in C([0, T], E) \times M^{p'}_w(C([0, T], E))\) we have
\begin{align}\label{eq: V bound}
\big| V(\omega, \nu) \big|^{p'/2} &\leq C \Big(1 + \int_s^t \big(\|\omega(r)\|_E^{p'} + \|\nu \circ \X_r^{-1}\|_{p'}^{p'} \big) dr\Big),
\end{align}
where we use \ref{A5}.
Since a.s. \(\mathcal{X}, \mathcal{X}^N \in M^{p'}_w (C([0, T], E))\), we have a.s.
\begin{align*}
Z (\mathcal{X}^N) &= \int V(\omega, \mathcal{X}^N) \mathcal{X}^N (d\omega) = \frac{1}{N} \sum_{i = 1}^N V(X^{N, i}, \mathcal{X}^N), \\ 
Z (\mathcal{X}) &= \int V(\omega, \mathcal{X}) \mathcal{X} (d\omega).
\end{align*}
Our next aim is to show that a.s. \(Z(\mathcal{X}) = 0\). Together with a monotone class argument, this implies that a.a. realizations of \(\mathcal{X}\) are \(p'\)-solution measures to the MKV SPDE \((A, \mu, \sigma, \eta)\).

\begin{lemma} \label{lem: K1}
	\(E\big[ |Z(\mathcal{X}^N)| \big] \to E\big[ |Z(\mathcal{X})| \big]\) as \(N \to \infty\).
\end{lemma}
\begin{proof}
The triangle inequality yields that
	\begin{equation} \label{eq: triangleq}
	\begin{split}
	\big|E\big[ |Z(\mathcal{X}^N)| \big] - E\big[ |Z(\mathcal{X})| \big]\big| \leq \big| E\big[ &|Z(\mathcal{X}^N)| \big] - E\big[ |Z_k(\mathcal{X}^N)| \big]\big|
	\\&+ \big| E\big[ |Z_k(\mathcal{X}^N)| \big] - E\big[ |Z_k(\mathcal{X})| \big]\big|
	\\&+ \big| E\big[ |Z_k(\mathcal{X})| \big] - E\big[ |Z(\mathcal{X})| \big]\big|.
	\end{split}
	\end{equation}
	Using \eqref{eq: V bound}, we estimate
	\begin{align*}
	E \big[\big | Z (\mathcal{X}^N) - Z_k (\mathcal{X}^N)\big| \big] &\leq \frac{1}{N}\sum_{i = 1}^N E \big[ \big| V(X^{N, i}, \mathcal{X}^N) - \big[(-k) \vee V(X^{N, i}, \mathcal{X}^N) \wedge k \big] \big| \big]
	\\&\leq \frac{1}{N}\sum_{i = 1}^N E \big[ \big| V(X^{N, i}, \mathcal{X}^N) \big| \1_{|V(X^{N, i}, \mathcal{X}^N)| > k}\big]
	\\&\leq \frac{1}{k^{p'/2 - 1}} \frac{1}{N}\sum_{i = 1}^N E \big[ \big| V(X^{N, i}, \mathcal{X}^N) \big|^{p'/2}\big]
	\\&\leq \frac{C}{k^{p'/2 - 1}} \Big(1 + \sup_{n \in \mathbb{N}}\sup_{r \leq t} E \big[ \|X^{n, 1}_r\|^{p'}_E \big]\Big), 
	\end{align*}
	where the constant \(C\) is independent of \(k\) and \(N\). By virtue of \eqref{eq: moment bound prop chaos pf}, this bound shows that the first term on the r.h.s. of \eqref{eq: triangleq} converges to zero as \(k \to \infty\) uniformly in \(N\). A similar computation shows the same claim for the final term. 
	Finally, the second term on the r.h.s. of \eqref{eq: triangleq} convergences to zero as \(N \to \infty\) because \(Z_k \in C_b (M^{p^\circ}_w(C([0, T], E)))\).
	The proof is complete.
\end{proof}
\begin{lemma} \label{lem: K2}
	\(E\big[\big(Z(\mathcal{X}^N)\big)^2\big] \to 0\) as \(N \to \infty\).
\end{lemma}
\begin{proof}
We compute
\begin{align*}
E \big[ \big(Z (\mathcal{X}^N)\big)^2 \big] 
&=  E \Big[ \Big(\int V(\omega, \mathcal{X}^N) \mathcal{X}^N(d\omega)\Big)^2\Big] \\&= \frac{1}{N^2} \sum_{i, j = 1}^N E \big[ V(X^{N, i}, \mathcal{X}^N)V(X^{N, j}, \mathcal{X}^N)\big].
\end{align*}
Passing to the analytically weak formulation of \(X^{N, i}\) and using It\^o's formula, we get that 
\[
M (X^{N, i}, \mathcal{X}^N) = \int_0^\cdot g'(\langle X^{N, i}_s, y^*\rangle_E ) d \Big(\int_0^s \langle \sigma^*(t, X^{N, i}_t, \mathcal{X}^N_t) y^*, d W^i_t\rangle_H\Big),
\]
see Step 4 in the proof of \Cref{theo: main} for more details. 
We obtain
\begin{align*}
[M (X^{N, i}, \mathcal{X}^N), M (X^{N, i}, \mathcal{X}^N)] 
&= \int_0^\cdot \big( g'(\langle X^{N, i}_s, y^*\rangle_E ) \big)^2 \| \sigma^*( s, X^{N, i}_s, \mathcal{X}^N_s) y^*\|^2_Hds 
\\&\leq \|g'\|_\infty^2 \|y^*\|_E^2 \int_0^\cdot \|\sigma (s, X^{N, i}_s, \mathcal{X}^N_s)\|^2_{L(H, E)} ds,
\end{align*}
where \([\cdot, \cdot]\) denotes the quadratic variation process.
Taking expectation and using \ref{A5}, we further get
\[
E \big[ [M (X^{N, i}, \mathcal{X}^N), M (X^{N, i}, \mathcal{X}^N)]_T \big] \leq C \Big(1 + E \Big[ \sup_{r \leq T} \|X^{N, 1}_r\|^{p'}\Big]^{2/p'}\Big) < \infty.
\]
Consequently, \(M (X^{N, i}, \mathcal{X}^N)\) is a square-integrable martingale. 
Therefore, we obtain that 
\begin{equation} \label{eq: mp M comp} \begin{split}
E\Big[ M_t (X^{N, i}&, \mathcal{X}^N) M_s (X^{N, j}, \mathcal{X}^N) \prod_{k = 1}^N h_k (X^{N, i}_{t_k}) h_k (X^{N, j}_{t_k}) \Big] \\&= E\Big[ M_s (X^{N, i}, \mathcal{X}^N) M_s (X^{N, j}, \mathcal{X}^N) \prod_{k = 1}^N h_k (X^{N, i}_{t_k}) h_k (X^{N, j}_{t_k}) \Big].
\end{split}
\end{equation}
Furthermore, for \(i \not = j\), as \(W^i\) and \(W^j\) are independent, \cite[Proposition A.3]{criens21} yields that
\[
\Big[ \int_0^\cdot \langle \sigma^*(t, X^{N, i}_t, \mathcal{X}^N_t) y^*, d W^i_t\rangle_H, \int_0^\cdot \langle \sigma^*(t, X^{N, j}_t, \mathcal{X}^N_t) y^*, d W^j_t\rangle_H\Big] = 0,
\]
which implies that
\begin{align*}
[M (X^{N, i}, &\mathcal{X}^N), M (X^{N, j}, \mathcal{X}^N)] = 0.
\end{align*}
Consequently, the product \(M (X^{N, i}, \mathcal{X}^{N}) M (X^{N, j}, \mathcal{X}^N)\) is a martingale  for \(i \not = j\) (see \cite[Proposition I.4.50]{JS}). 
We therefore obtain for \(i \not = j\) that
\begin{align*}
E\Big[ M_t (X^{N, i}&, \mathcal{X}^N) M_t (X^{N, j}, \mathcal{X}^N) \prod_{k = 1}^N h_k (X^{N, i}_{t_k}) h_k (X^{N, j}_{t_k}) \Big] \\&= E\Big[ M_s (X^{N, i}, \mathcal{X}^N) M_s (X^{N, j}, \mathcal{X}^N) \prod_{k = 1}^N h_k (X^{N, i}_{t_k}) h_k (X^{N, j}_{t_k}) \Big].
\end{align*}
Together with \eqref{eq: mp M comp}, we deduce that
\begin{align*}
E \big[V(X^{N, i}, \mathcal{X}^N) V(X^{N, j}, \mathcal{X}^N)\big] = 0, \qquad i \not = j.
\end{align*}
In summary, using \eqref{eq: V bound}, we obtain 
\begin{align*}
\frac{1}{N^2} \sum_{i, j = 1}^N E \big[ V(X^{N, i}, \mathcal{X}^N)V(X^{N, j}, \mathcal{X}^N)\big] &= \frac{1}{N^2}\sum_{i = 1}^N E \big[ \big(V(X^{N, i}, \mathcal{X}^N)\big)^2 \big]
\\&\leq  \frac{C }{N}\Big(1 + \sup_{n \in \mathbb{N}}E \Big[ \sup_{r \leq t} \|X^{n, 1}_r\|^{p'}\Big]^{2/p'}\Big).
\end{align*}
As the final term converges to zero as \(N \to \infty\), the claim of the lemma follows.
\end{proof}
Combing \Cref{lem: K1,lem: K2}, we obtain that 
\[
E \big[ \big| Z(\mathcal{X})\big| \big] = \lim_{N\to \infty} E \big[ \big| Z(\mathcal{X}^N)\big| \big] \leq \lim_{N \to \infty} E \big[ \big( Z(\mathcal{X}^N)\big)^2 \big]^{1/2} = 0,
\]
which implies that a.s. \(Z(\mathcal{X}) = 0\).
\\

\noindent
\emph{Step 4: Identifying the Limit.} We are now in the position to identify the limit \(\mathcal{X}\).  
However, to use our assumption \ref{UL} we first have to adjust our setting to the infinite time horizon. By Step~1 and the \cite[Theorem 23.4]{Kallenberg}, the family \(\{\mathcal{X}^N \colon N \in \mathbb{N}\}\) is tight when considered as random variables in \(M_c(C(\mathbb{R}_+, E))\). Thus, there exists an accumulation point \(\mathcal{X}^*\) such that \(\mathcal{X}^*|_{[0, T]} = \mathcal{X}|_{[0, T]}\) in law. In the following we show that a.s. \(\mathcal{X}^* = \mathcal{X}^0\) where \(\mathcal{X}^0\) denotes the unique law of a \(p'\)-solution process of the MVK SPDE  \((A, \mu, \sigma, \eta)\). This then implies that \(\mathcal{X}^N|_{[0, T]} \to \mathcal{X}^0|_{[0, T]}\) in probabilty when seen as random variables in \(M_w^{p^\circ}(C([0, T], E))\). The final claim of the theorem will then follow from Vitali's theorem.

\begin{lemma}\label{lem: approx D(A)}  
	There exists a countable set \(D \subset D(A^*)\) such that for every \(y^* \in D(A^*)\) there exists a sequence \(y^*_1, y^*_2, \dots \in D\) with \(y^*_n \to y^*\) and \(A^* y^*_n \to A^* y^*\). \end{lemma}
\begin{proof}
As \(E\) is assumed to be separable, there exists a countable dense subset \(F\). 
As \(A\) generates a \(C_0\)-semigroup, there exists a \(\lambda > 0\) in the resolvent set of \(A\) and hence also in the resolvent set of \(A^*\), see \cite[Theorem~1.5.3, Lemma 1.10.2]{pazy}.
Now, set
\[
D \triangleq \big\{ (\lambda - A^*)^{-1} x \colon x \in F\big\} \subset D(A^*). 
\]
Obviously, \(D\) is countable. Let \(y^* \in D(A^*)\) and set \(x \triangleq \lambda y^* - A^* y^* = (\lambda - A^*) y^*\). As \(F\) is dense, there exists a sequence \(x_1, x_2, \dots \in F\) such that \(x_n \to x\). Now, set \(y^*_n \triangleq (\lambda - A^*)^{-1} x_n \in D\) for \(n \in \mathbb{N}\). We have \(y^*_n \to (\lambda - A^*)^{-1} x = y^*\) as \((\lambda - A^*)^{-1} \in L(E)\). Furthermore, \(- A^* y^*_n = (\lambda - A^*) y^*_n - \lambda y^*_n = x_n - \lambda y^*_n \to x - \lambda y^* = - A^* y^*\). This shows the claim.
\end{proof}

Let \(\mathcal{C} \subset C_b(E)\) be a countable set which is measure determining (\cite[Proposition 3.4.2]{EK}).
Let \(\mathcal{G} \triangleq \{g^n_1, g^n_2 \colon n \in \mathbb{N}\}\), where \(g^n_1, g^n_2 \in C_c^2(\mathbb{R})\) are such that \(g^n_1 (x) = x\) and \(g^n_2(x) = x^2\) for \(|x| \leq n\).

We realize \(\mathcal{X}^*\) on a probability space \((\Omega, \mathcal{F}, P)\). 
Let \(G\) be the set of all \(\omega \in \Omega\) such that \(\mathcal{X}^* |_{[0, M]}(\omega) \in M^{p'}_w (C([0, M], E))\) for all \(M \in \mathbb{N}\), \(\mathcal{X}^*_0 (\omega) = \eta\) and \(Z(\mathcal{X}^*(\omega)) = 0\) for all \(s, t \in \mathbb{Q}_+,\)  \(s < t, y^* \in D, g \in \mathcal{G}, m \in \mathbb{N}, t_1, \dots, t_m \in [0, s] \cap \mathbb{Q}_+\) and \(h_1, \dots, h_m \in \mathcal{C}\).
\begin{lemma}
	\(P(G) = 1\).
\end{lemma}
\begin{proof}
	Note that a.s. \(\mathcal{X}^*_0 = \eta\) thanks to assumption \ref{I}. 
Fix \(M \in \mathbb{N}\) and suppose that \(\mathcal{X}^{N_m} \to \mathcal{X}^*\) in \(M_c(M_c(C(\mathbb{R}_+, E)))\) as \(m \to \infty\). Recalling \eqref{eq: moment ineq final}, we get a.s. \(\mathcal{X}^*|_{[0, M]} \in M^{p'}_w(C([0, M], E))\).
By Step~2,  there exists a subsequence of \(\{\mathcal{X}^{N_m}|_{[0, M]} \colon m \in \mathbb{N}\}\) which converges to a limit \(\mathcal{X}^\circ\) in \(M_c(M^{p^\circ}_w(C([0, M], E)))\) and we have \(\mathcal{X}^*|_{[0, M]} = \mathcal{X}^\circ\) in law. 
Now, Step~3 yields that a.s. \(Z(\mathcal{X}^*) = 0\), where \(Z\) is defined with \(T = M\). In summary, \(G\) is the intersection of countably many full sets and therefore a full set by itself. This is the claim. 
\end{proof}

Take \(\omega \in G\). For every \(y^*\in D\), a monotone class argument and similar considerations as in Step~5 of the proof for \Cref{theo: main} show that 
\begin{align*}
\langle \X, y^*\rangle_E - \langle \X_0, y^*\rangle_E- \int_0^\cdot \big( \langle \X_s, A^* y^*\rangle_E + \langle \mu (s, \X_s, \mathcal{X}^*_s(\omega)), y^* \rangle_E \big) ds
\end{align*}
is a local \(\mathcal{X}^*(\omega)\)-martingale (on the space \(C(\mathbb{R}_+, E)\) endowed with the natural filtration generated by the coordinate process \(\X\)) with quadratic variation process
\(
\int_0^\cdot \| \sigma^* (s, \X_s, \mathcal{X}^*_s(\omega)) y^*\|^2_H ds.
\)
Using \Cref{lem: approx D(A)}  and the fact that ucp (uniformly on compacts in probability) limits of continuous local martingales are again continuous local martingales, we can conclude that the same holds for all \(y^* \in D(A^*)\).
As \(D(A^*)\) separates points, the representation theorem \cite[Theorem 3.1]{doi:10.1142/S0219025707002816} and the equivalence of the weak and mild formulation as given by \cite[Theorem~13]{MR2067962} yield that \(\mathcal{X}^*(\omega)\) is a \(p'\)-solution measure of the MKV SPDE~\((A, \mu, \sigma, \eta)\).

Consequently, by assumption \ref{UL}, a.s. \(\mathcal{X}^* = \mathcal{X}^0\).
We conclude that \(\mathcal{X}^N |_{[0, T]} \to \mathcal{X}^0|_{[0, T]}\) in \(M_c(M_w^{p^\circ}(C([0, T], E)))\) as \(N \to \infty\), and therefore also \(\mathsf{w}^{p^\circ}_T (\mathcal{X}^N, \mathcal{X}^0) \to 0\) in probability, where \(\mathsf{w}^{p^\circ}_T\) denotes the \(p^\circ\)-Wasserstein metric on \(M^{p^\circ}_w(C([0, T], E))\).

Let us show that the family \(\{|\mathsf{w}^{p^\circ}_T (\mathcal{X}^N, \mathcal{X}^0)|^{p^\circ} \colon N \in \mathbb{N}\}\) is uniformly integrable. In this case, \eqref{eq: new WS metric conv} follows from Vitali's theorem. We estimate
\begin{align*}
	E \big[ \big(\big|\mathsf{w}^{p^\circ}_T (\mathcal{X}^N, \delta_0)\big|^{p^\circ} \big)^{(p'/p^\circ)} \big] &= E \Big[ \Big( \frac{1}{N} \sum_{i = 1}^N \sup_{s \leq T} \|X^{N, i}_s\|^{p^\circ}_E \Big)^{p'/p^\circ} \Big]
	\\&\leq E \Big[ \frac{1}{N} \sum_{i = 1}^N \sup_{s \leq T} \|X^{N, i}_s\|^{p'}_E \Big] 
	\\&= E \Big[ \sup_{s \leq T} \|X^{N, 1}_s\|^{p'}_E\Big],
\end{align*}
where we use H\"older's inequality in the second line. Since \(p'/p^\circ > 1\) and \(\mathcal{X}^0 \in M^{p'}_w (C([0, T], E))\), \eqref{eq: moment bound prop chaos pf} yields that the family \(\{|\mathsf{w}^{p^\circ}_T (\mathcal{X}^N, \mathcal{X}^0)|^{p^\circ} \colon N \in \mathbb{N}\}\) is uniformly integrable. Consequently, \eqref{eq: new WS metric conv} holds.
Finally, the chaotic property follows from \cite[Proposition I.2.2]{10.1007/BFb0085169}. 
\qed

\appendix 

\section{An Existence and Uniqueness Result for classical SPDEs}
Let \(\mu \colon \mathbb{R}_+ \times E \to E\) and \(\sigma \colon \mathbb{R}_+ \times E \to L(H, E)\) be Borel functions, and take a constant \(0 < \alpha < 1/2\).
The following theorem should be compared to \cite[Theorem 7.6]{DePratoEd1}. Its proof follows the standard path but for completeness we outline the argument.
\begin{theorem} \label{theo: DZ}
	Assume that for every \(T > 0\) there exist Borel functions \(\f = \f_T \colon (0, T] \to [0, \infty]\) and \(\g = \g_T \colon (0, T] \to [0, \infty]\) such that 
	\[
	\int_0^T \Big(\Big[\frac{\f (s) }{s^{\alpha}} \Big]^2 + \big[ \g (s) \big]^{p / ( p - 1 )} \Big) ds < \infty,
	\] 
	and 
	\begin{align*}
\|S_t(\sigma (s, x) - \sigma (s, y))\|_{L_2(H, E)} &\leq \f (t) \|x - y\|_E,\\
	\|S_t (\mu (s, x) - \mu (s, y))\|_E &\leq \g (t) \|x - y\|_E , \\
 \|S_t \sigma (s, x)\|_{L_2(H, E)} &\leq \f (t) (1 + \|x\|_E), \\
 	\|S_t \mu (s, x)\|_E &\leq \g (t) (1 + \|x\|_E), 
	\end{align*}
	for all \(0 < t, s \leq T\) and \(x, y \in E\). Then, for any \(\eta \in M_c(\eta)\), on any driving system \((\B,W)\) there exists a unique, up to indistinguishability, continuous mild solution process \(X\) to the SPDE
	\[
	d X_t = A X_t dt + \mu (t, X_t) dt + \sigma (t, X_t) d W_t, \quad X_0 \sim \eta.
	\]
Moreover, for every \(p > 1/\alpha, T> 0\) and \(\eta \in M^p_w(E)\), 
\[
E \Big[ \sup_{s \leq T} \|X_s\|^p_E \Big] < \infty.
\]
\end{theorem}
Here, a mild solution is meant to be defined in the usual sense, i.e. similar to \Cref{def: SPDE classical} without the coefficient \(\gamma\).

\begin{proof}[Sketch of Proof]
	We start by proving the second part of the theorem. Let \(p > 1 / \alpha\) and assume that \(\eta \in M^p_w (E)\).
	Take a completed filtered probability space \((\Omega, \cF, (\cF_t)_{t \geq 0}, P)\) which supports a cylindrical Brownian motion \(W\) and an \(\cF_0\)-measurable \(\xi\) such that \(\xi \sim \eta\).
	Moreover, define \(\cH^p\) to be the space of continuous \(E\)-valued processes \(Y = (Y_t)_{t \geq 0}\) such that
	\[
	E \Big[ \sup_{s \leq T} \|Y_s\|^p_E \Big] < \infty, \quad \forall T > 0.
		\]
		Define a map \(I \colon \cH^p \to \cH^p\) by 
		\[
		I (Y) (t) \triangleq S_t \xi + \int_0^t S_{t - s} \mu (s, Y_s) ds + \int_0^t S_{t - s} \sigma (s, Y_s) d W_s, \quad t \in\mathbb{R}_+.
		\]
		Let us elaborate in more detail that \(I (\cH^p) \subset \cH^p\). First of all, 
		\[
		E \Big[ \sup_{s \leq T} \|I (Y) (s)\|^p_E\Big] < \infty, \quad \forall T > 0, 
		\]		
		follows from the linear growth assumptions, \Cref{lem: factorization} and the estimate 
		\begin{align*}
		E \Big[ \sup_{s \leq T} \Big\| \int_0^t S_{t - s} \mu (s, Y_s) ds \Big\|^p_E \Big] &\leq E \Big[ \sup_{s \leq T} \Big( \int_0^t \g (t - s) (1 + \|Y_s\|_E) ds \Big)^p \Big] 
		\\&\leq E \Big[ \Big(\int_0^T \big[\g (s)\big]^{p / (p - 1)} ds \Big)^{p - 1} \int_0^T ( 1 + \|Y_s\|_E )^p ds \Big]
		\\&\leq \Big(\int_0^T \big[\g (s)\big]^{p / (p - 1)} ds \Big)^{p - 1} 2^{p + 1}T \Big( 1 +  E \Big[ \sup_{s \leq T} \|Y_s\|^p_E \Big] \Big),
		\end{align*}
		which uses H\"older's inequality. While the first two terms in the definition of \(I (Y)\) are clearly continuous (\cite[Lemma~6.2.9]{roeckner15}), the last integral is continuous thanks to \Cref{lem: factorization}. This shows that \(I (\cH^p) \subset \cH^p\). Next, for \(Y, Z \in \cH^p\) and \(T > 0\) set 
		\[
		\Phi_T (Y, Z) \triangleq E \Big[ \sup_{s \leq T} \|Y_s - Z_s\|^p_E \Big].
		\]
		For every \(T > 0\), using the Lipschitz hypothesis and similar arguments are above, we obtain 
		\begin{align}\label{eq: contraction type ineq}
		\Phi_T ( I (Y), I (Z) ) \leq C \int_0^T E \big[ \|Y_s - Z_s \|_E^p \big] ds \leq C \int_0^T \Phi_s ( Y, Z ) ds,
		\end{align}
		where the constant depends on \(T, \g = \g_T\) and \(\f = \f_T\). Define now inductively a sequence \(X^0, X^1,\) \(X^2, \dots\) such that \(X^0 \triangleq S \xi\) and \(X^n \triangleq  I (X^{n - 1})\) for \(n = 1, 2, \dots\). It follows from \eqref{eq: contraction type ineq} and induction that 
		\[
		\Phi_T(X^{n - 1}, X^n) \leq \frac{C^n T^n}{n!} \Phi_T(X^0, X^1), \quad n = 1, 2, \dots, \ \ T > 0.
		\]
		Consequently, by a Borel--Cantelli argument (\cite[Theorem 5.2.9]{KaraShre}), we deduce that the sequence \(X^1, X^2, \dots\) converges a.s. in the local uniform topology to a continuous process \(X\).
		Furthermore, we obtain that 
		\begin{align*}
		E \Big[ \sup_{n \in \mathbb{N}} \sup_{s \leq T} \|X^n_s - \xi\|^p_E \Big] &\leq \sum_{n = 1}^\infty n^{p + 1} E \Big[ \sup_{s \leq T} \|X^n_s - X^{n - 1}_s\|^p_E \Big] 
		\\&\leq \sum_{n = 1}^\infty n^{p + 1} \frac{C^n T^n}{n!} \Phi_T(X^0, X^1) < \infty.
		\end{align*}
		Thus, \(X \in \cH^p\) and the dominated convergence theorem together with \eqref{eq: contraction type ineq} yield that 
		\[
		\Phi_T (I (X), X) = \lim_{n \to \infty} \Phi_T(I(X), X^n) \leq \lim_{n \to \infty} C \int_0^T \Phi_s (X, X^n) ds = 0.
		\]
		We conclude that a.s. \(I (X) = X\), which shows that \(X\) is a continuous mild solution process. 
		Furthermore, \eqref{eq: contraction type ineq} and Gronwall's lemma yield uniqueness up to indistinguishability. 
		
		Finally, let us comment on the general case where \(\eta \in M_c(E)\). For \(m = 1, 2, \dots\), let \(X^m\) be a solution as constructed above for the initial value \(\xi^m \triangleq \xi \1_{\|\xi\|_E \leq m}\). Then, a.s. \(X^n = X^{n + 1}\) on \(\{\|\xi\|_E \leq n\}, n = 1, 2, \dots\). Consequently, \(\lim_{n \to \infty} X^n\) is a.s. well-defined and a continuous mild solution.
\end{proof}
\section{Proof of \Cref{theo: Lip}} \label{sec: pf theo Lip}
By virtue of \Cref{theo: uni}, it suffices to prove the existence of a \(p\)-solution process on any given driving system \((\B, W)\). 
We use a classical argument based on a fixed point theorem (see \cite[Theorem 4.21]{Carmona18} for the argument in finite dimensions with finite time horizon). In the following we show existence on a finite time interval \([0, T]\) with a random initial value \(\xi_0 \sim \eta\). The existence of a global solution follows from the local result by pasting.
Let \(\gamma \in C([0, T], M^p_w (E))\). By \Cref{theo: DZ} (in case \(p > 1/\alpha\) and \ref{L1} hold) and \cite[Theorem 3.3]{gawarecki2010stochastic} (in case \(p \geq 2\), \ref{L2} holds and \(S\) is a generalized contraction), there exists a continuous mild solution process \(X^{(\gamma)}\) to the SPDE 
\[
d X_t = A X_t dt+ \mu(t, X_t, \gamma_t) d t + \sigma (t, X_t, \gamma_t) d W_t, \quad X_0 = \xi_0, 
\]
such that
\[
E \Big[ \sup_{s \leq T} \|X^{(\gamma)}_s\|_E^p\Big] < \infty.
\]
Furthermore, by the Yamada--Watanabe theorem \cite[Theorem 2]{MR2067962}, the law of \(X^{(\gamma)}\) is fully characterized by \(A, \mu, \sigma, \eta\) and \(\gamma\).
We now define a map \(\Phi \colon C([0, T], M^p_w (E)) \to C([0, T], M^p_w (E))\) by 
\[
\Phi (\gamma) (t) = P \circ (X^{(\gamma)}_t)^{-1}, \quad t \leq T.
\]
Let \(\gamma, \gamma' \in C([0, T], M^p_w (E))\). 
As in the proof of \Cref{lem: uni main tech}, we obtain the estimate
\begin{align*}
E \big[ \|X^{(\gamma)}_t - X^{(\gamma')}_t \|^p_E \big] &\leq C \Big(\int_0^t E \big[ \|X_s^{(\gamma)} - X_s^{(\gamma')}\|^p_E\big] ds + \int_0^t |\mathsf{w}_p (\gamma (s), \gamma'(s))|^p ds\Big), \quad t \leq T.
\end{align*}
Thus, Gronwall's lemma yields that 
\[
|\mathsf{w}_p(\Phi(\gamma) (t), \Phi(\gamma') (t))|^p \leq E \big[ \|X^{(\gamma)}_t - X^{(\gamma')}_t \|^p_E \big] \leq C \int_0^t |\mathsf{w}_p (\gamma (s), \gamma'(s))|^p ds.
\]
Induction yields for every \(k \in \mathbb{N}\) that
\begin{align*}
\sup_{s \leq T} |\mathsf{w}_p(\Phi^k(\gamma) (s),\Phi^k(\gamma') (s))|^p  &\leq C^k \int_0^T \frac{(T - s)^{(k - 1)}}{(k - 1)!} |\mathsf{w}_p (\gamma (s), \gamma'(s))|^p ds
\\&\leq   \frac{C^kT^k}{k!} \sup_{s \leq T}|\mathsf{w}_p (\gamma (s), \gamma'(s))|^p.
\end{align*}
Thus, there exists an \(N \in \mathbb{N}\) such that \(\Phi^k\) is a contraction on the Polish space \(C([0, T], M^p_w(E))\) for all \(k \geq N\). Thanks to the theorem in \cite{Bryant1968ARO}, this yields that \(\Phi\) has a fixed point and consequently, restricted to the time interval \([0, T]\) there exists a solution process to the MKV SPDE with coefficients  \((A, \mu, \sigma, \eta)\). 

Finally, the existence of a global solution process, i.e. a solution process defined on the infinite time interval \(\mathbb{R}_+\), follows by pasting: Let \(X^0 \triangleq \xi_0\) and for \(n \in \mathbb{N}\) let \(X^n\) be a solution process with coefficients \(\mu (\cdot + n - 1, \cdot, \cdot), \sigma (\cdot + n - 1, \cdot, \cdot)\), initial value \(X^{n - 1}_{n - 1}\) and with driving noise \(W^n \triangleq W_{\cdot + n - 1} - W_{n - 1}\). Finally, define
\[
X_t \triangleq \sum_{k = 1}^\infty X^k_{t - (k - 1)} \1_{k - 1 \leq t < k}, \qquad t \in \mathbb{R}_+.
\]
For \(k < t \leq k + 1\) we compute that
\begin{align*}
X_t &= S_{t - k}X^k_k + \int_0^{t - k} S_{t - k - s} \mu (s + k, X^{k + 1}_s, P^{X^{k + 1}}_s) ds \\&\hspace{3cm}+ \int_0^{t - k} S_{t - k - s} \sigma (s + k, X^{k + 1}_s, P^{X^{k + 1}}_s) d W^{k + 1}_s
\\&= S_{t - k}X^k_k + \int_k^{t} S_{t - s} \mu (s, X^{k + 1}_{s - k}, P^{X^{k + 1}}_{s - k}) ds + \int_k^{t} S_{t -s} \sigma (s, X^{k + 1}_{s - k}, P^{X^{k + 1}}_{s - k}) d W_s
\\&= S_{t - k}X^k_k + \int_k^{t} S_{t - s} \mu (s, X_{s}, P^{X}_{s}) ds + \int_k^{t} S_{t -s} \sigma (s, X_{s}, P^{X}_{s}) d W_s.
\end{align*}
Thus, by induction, \(X\) is a solution process to the MKV SPDE with coefficients \((A, \mu, \sigma, \eta)\).
\qed

\section{Proof of \Cref{prop: weakly interact existence}} \label{sec: pf prop ex}
Fix \(N \in \mathbb{N}\) and consider the (separable) Hilbert spaces
\(
\widetilde{E} \triangleq \bigoplus_{i = 1}^N E\) and \(\widetilde{H} \triangleq \bigoplus_{i = 1}^N H.\)
Moreover, for \(t \in \mathbb{R}_+\) and \(e = (e^1, \dots, e^N) \in \widetilde{E}\) we set 
\begin{align*}
L (e) &\triangleq \frac{1}{N} \sum_{i = 1}^N \delta_{e^i},\\
\widetilde{\mu} (t, e) &\triangleq \bigoplus_{i = 1}^N \mu (t, e^i, L (e)) \in \widetilde{E}, \\ \widetilde{\sigma} (t, e) &\triangleq \bigoplus_{i = 1}^N \sigma (t, e^i, L(e)) \in L (\widetilde{H}, \widetilde{E}).
\end{align*}
It is not hard to check that the process \(\widetilde{W} \triangleq \bigoplus_{i = 1}^N W^i\) is a standard cylindrical Brownian motion. The system of SPDEs associated to the processes \(X^{N, 1}, \dots, X^{N, N}\) can now be written as
\[
d \widetilde{X}_t = \widetilde{A} \widetilde{X}_t dt + \widetilde{\mu}(t,\widetilde{X}_t) dt + \widetilde{\sigma} (t, \widetilde{X}_t) d \widetilde{W}_t, \quad \widetilde{A} \triangleq \bigoplus_{i = 1}^N A,
\]
where \(\widetilde{A}\) generates the \(C_0\)-semigroup \(\widetilde{S} \triangleq \bigoplus_{i = 1}^N S\) on \(\widetilde{E}\).
Thus, by virtue of \Cref{theo: DZ}, the claim of the proposition follows in case \(\widetilde{\mu}\) and \(\widetilde{\sigma}\) satisfy suitable Lipschitz and linear growth conditions, which we check in the following. Take \(T > 0\) and let \(\f = \f_T\) be as in \ref{L1}.
Then, for all \(0 < t, s \leq T\) and \(e = (e^1, \dots, e^N) \in \widetilde{E}\) we get
\begin{align*}
\|\widetilde{S}_t \widetilde{\sigma} (s, e, L (e)) \|_{L_2 (\widetilde{H}, \widetilde{E})}^2 &= \sum_{i = 1}^N \| S_t \sigma (s, e^i, L(e))\|_{L_2(E, H)}^2
\\&\leq \big[\f (t)\big]^2 \sum_{i = 1}^N C \big(1 + \|e_i\|^2_E + \|L(e)\|_2^2\big)
= \big[\f (t)\big]^2\big(CN + 2C \|e\|_{\widetilde{E}}^2\big).
\end{align*}
Similarly, we obtain for all \(0 < t, s \leq T\) and \(e = (e^1, \dots, e^N), f = (f^1, \dots, f^N) \in \widetilde{E}\) that
\begin{align*}
\|\widetilde{S}_t (\widetilde{\sigma} (s, e, L(e)) - \widetilde{\sigma} (t, f, L(f)))\|^2_{L_2(\widetilde{H}, \widetilde{E})} 
\leq \big[\f (t)\big]^2 \big( C \|e - f\|_{\widetilde{E}}^2 + C N |\mathsf{w}_2( L(e) , L(f))|^2 \big).
\end{align*}
We now compute \(\mathsf{w}_2(L(e), L(f))\). Set 
\[
F \triangleq \frac{1}{N} \sum_{i = 1}^N \delta_{(e_i, f_i)}.
\]
Then, \(F(dx \times E) = L(e) (dx), F(E \times dx) = L(f)(dx)\) and
\[
\iint \|x - y\|^2_E F(dx, dy) = \frac{1}{N}\sum_{i = 1}^N \|e_i - f_i\|^2_{E} = \frac{1}{N} \|e - f\|_{\widetilde{E}}^2.
\]
Hence, we obtain
\[
|\mathsf{w}_2(L(e), L(f))|^2 \leq \frac{1}{N} \|e - f\|_{\widetilde{E}}^2,
\]
and finally, 
\[
\|\widetilde{S}_t (\widetilde{\sigma} (s, e, L(e)) - \widetilde{\sigma} (t, f, L(f)))\|^2_{L_2(\widetilde{H}, \widetilde{E})} \leq 2 C \big[\f (t)\big]^2 \|e - f\|_{\widetilde{E}}^2.
\]
We conclude that the coefficient \(\widetilde{\sigma}\) satisfies the linear growth and Lipschitz conditions from \Cref{theo: DZ}. Similar computations show the same for the coefficient \(\widetilde{\mu}\). We omit the remaining details. \qed

\section{Proof of \Cref{theo: POC}} \label{sec: pf POC lip}
We borrow the main idea from the proof of \cite[Theorem~3.3]{Lak18}. By virtue of \Cref{theo: Lip}, let \(Y^i\) be a \(p\)-solution process to the MKV SPDE \((A, \mu, \sigma, \eta)\) on the driving system \((\B, W^i)\) with initial value \(\xi_0^i\). Take \(T > 0\) and denote by \(\mathcal{X}^0_t\) the projection of \(\mathcal{X}^0\) to the time \(t\) value. 
By virtue of the proof of \Cref{lem: uni main tech}, using \Cref{lem: factorization} (in case \(p > 1/\alpha\) and \ref{L1} hold) and part (b) of \cite[Lemma~3.3]{gawarecki2010stochastic} (in case \(p \geq 2\), \ref{L2} holds and \(S\) is a generalized contraction), for all \(t \leq T\) we obtain that  
\begin{align*}
E \Big[ \sup_{s \leq t} \|X^{N, i}_s - Y^i_s \|_E^p \Big] 
&\leq C E \Big[ \int_0^t \big( \|X^{N, i}_s - Y^i_s\|^p_E + |\mathsf{w}_p (\mathcal{X}^N_s, \mathcal{X}^0_s)|^p\big) ds \Big]
\\&\leq  C E \Big[ \int_0^t \Big( \sup_{r \leq s} \|X^{N, i}_r - Y^i_r\|^p_E + |\mathsf{w}_p (\mathcal{X}^N_s, \mathcal{X}^0_s)|^p\Big) ds \Big].
\end{align*}
Thus, Gronwall's lemma yields that 
\begin{align}\label{eq: first gronwall}
E \Big[ \sup_{s \leq t} \|X^{N, i}_s - Y^i_s \|_E^p \Big] \leq C \int_0^t E \big[|\mathsf{w}_p (\mathcal{X}^N_s, \mathcal{X}^0_s)|^p\big] ds, \quad t \leq T.
\end{align}
We set 
\[
\mathcal{Y}^N \triangleq \frac{1}{N} \sum_{i = 1}^N \delta_{Y^i}.
\]
Using the coupling \(\frac{1}{N} \sum_{i = 1}^N \delta_{(X^{N, i}, Y^i)}\), we obtain that
\[
|\mathsf{w}_t^p (\mathcal{X}^N, \mathcal{Y}^N)|^p \leq \frac{1}{N} \sum_{i = 1}^N \sup_{s \leq t}\|X^{N, i}_s - Y^i_s\|^p_E.
\]
Hence, for all \(t \leq T\)
\begin{align*}
E \big[ |\mathsf{w}_t^p (\mathcal{X}^N, \mathcal{X}^0)|^p \big] &\leq C \big(E \big[ |\mathsf{w}_t^p (\mathcal{X}^N, \mathcal{Y}^N)|^p \big] + E \big[ |\mathsf{w}_t^p (\mathcal{Y}^N, \mathcal{X}^0)|^p \big]\big)
\\&\leq C \Big(\int_0^t E \big[|\mathsf{w}_p (\mathcal{X}^N_s, \mathcal{X}^0_s)|^p\big] ds + E \big[ |\mathsf{w}_t^p (\mathcal{Y}^N, \mathcal{X}^0)|^p \big]\Big)
\\&\leq C\Big( \int_0^t E \big[|\mathsf{w}_s^p (\mathcal{X}^N, \mathcal{X}^0)|^p\big] ds + E \big[ |\mathsf{w}_t^p (\mathcal{Y}^N, \mathcal{X}^0)|^p \big]\Big).
\end{align*}
Using Gronwall's lemma once again, we conclude that 
\begin{align*}
E \big[ |\mathsf{w}_T^p (\mathcal{X}^N, \mathcal{X}^0)|^p \big] \leq C E \big[ |\mathsf{w}_T^p (\mathcal{Y}^N, \mathcal{X}^0)|^p \big].
\end{align*}
The r.h.s. converges to zero as \(N \to \infty\) by \cite[Corollary 2.14]{Lak18}, as \(W^1, W^2, \dots\) and the initial values are i.i.d. and so are \(Y^1, Y^2, \dots\). Thus, \eqref{eq: first conv} is proved.

Let us now show the second claim, namely \eqref{eq: second conv}. We deduce from \eqref{eq: first conv} and \eqref{eq: first gronwall} that
\[
E \Big[ \max_{i = 1, \dots, k} \sup_{s \leq T} \|X^{N, i}_s - Y^i_s \|_E^p \Big] \leq C k E \big[ |\mathsf{w}_T^p (\mathcal{X}^N, \mathcal{X}^0)|^p \big]\to 0 
\]
as \(N \to \infty\). This immediately implies \eqref{eq: second conv}. The proof is complete. \qed

%

\bibliographystyle{plain}

\end{document}